\documentclass[reqno,12pt]{amsart}
\usepackage{bbm}
\usepackage[all,pdf]{xy}
\usepackage{epsfig}
\usepackage{amsmath}
\usepackage{amssymb}
\usepackage{amscd}
\usepackage{graphicx}
\usepackage{enumerate}
\usepackage{comment}
\usepackage[colorlinks=true]{hyperref}
\hypersetup{urlcolor=red, citecolor=blue}
\makeatletter
\@namedef{subjclassname@2020}{%
	\textup{2020} Mathematics Subject Classification}

\newtheorem{thm}{Theorem}[section]
\newtheorem{lem}[thm]{Lemma}

\newtheorem{cl}[thm]{Claim}
\newtheorem{prop}[thm]{Proposition}

\newtheorem{cor}[thm]{Corollary}
\newtheorem{defn}[thm]{Definition}

\newtheorem{rem}[thm]{Remark}

\newtheorem{conj}[thm]{Conjecture}

\def \N {\mathbb N}

\def \Z {\mathbb Z}

\def\A {\mathcal A}

\def\V {\mathcal V}

\def\X {\mathcal{X}}
\def\F {\mathcal{F}}
\numberwithin{equation}{section}

\def\*{(*)}

\begin{document}
	\title[Independence  and mean sensitivity]{Independence and mean sensitivity in minimal systems under group actions}
	
	\author{Chunlin Liu}
	\address[C. Liu]{School of Mathematical Sciences, Dalian University of Technology, Dalian, 116024, P.R. China}
	\email{chunlinliu@mail.ustc.edu.cn}	
	
	\author{Leiye Xu}
	\address[L. Xu]{School of Mathematical Sciences, University of Science and Technology of China, Hefei, Anhui, 230026, P.R. China}
	\email{leoasa@mail.ustc.edu.cn}
	
	\author{Shuhao Zhang}
	\address[S. Zhang]{School of Mathematical Sciences, University of Science and Technology of China, Hefei, Anhui, 230026, P.R. China}
	\email{yichen12@mail.ustc.edu.cn}

	\subjclass[2020]{Primary: 37B05, 37A15}

	\keywords{Independence set, IT-tuples, regular extension, mean sensitivity}
	
	\thanks{}
	\begin{abstract}
		 In this paper, we mainly study the relation between regularity, independence  and mean sensitivity for minimal systems.
		In the first part, we show that if a minimal system is incontractible, or  local Bronstein with an invariant Borel probability measure, then the regularity  is strictly bounded by the infinite independence. In particular, the following two types of minimal systems are applicable to our result:
		\begin{enumerate}
			\item The acting group of the minimal system is a virtually nilpotent group.
			\item The minimal system is a proximal extension  of its  maximal equicontinuous factor and admits an invariant Borel probability measure. 
		\end{enumerate}
		Items~(1) and~(2) correspond to Conjectures~1 and~2 from Huang, Lian, Shao, and Ye (J.~Funct.~Anal., 2021); item~(1) verifies Conjecture~1 in the virtually nilpotent case, and item~(2) gives an affirmative answer to Conjecture~2.
		
		\medskip
		
		In the second part,	for a minimal system acting by an amenable group, under the local Bronstein condition, we  establish  parallel results regarding weak mean sensitivity and establish that every mean-sensitive tuple is an IT-tuple. 
	\end{abstract}
	
	\maketitle
	\section{Introduction}
	\subsection{Background}
	Throughout this paper, we assume that $G$ is an infinite countable discrete group. By a topological dynamical system (abbreviated as tds), we mean a pair $(X, G)$, where $G$ acts on a compact metric space $X$ by homeomorphisms. A (Borel) probability measure $\mu$ on $X$ is said to be ($G$-)invariant if $\mu(gA) = \mu(A)$ for all $g \in G$ and all Borel subsets $A$ of  $ X$ and it is ergodic if $\mu$ is invariant and $\mu\left(\cup_{g \in G} gA\right) \in \{0, 1\}$ for all Borel subsets $A $ of $ X$.

	In \cite{B1}, Blanchard introduced the concept of entropy pairs as a tool to identify where entropy resides within a tds, and demonstrated that a tds has positive entropy if and only if it possesses essential entropy pairs. This approach, now widely known as local entropy theory (see \cite{GH, GY} for surveys), has become a valuable tool of entropy theory.
	Building on this foundation, Blanchard et al. \cite{B2} extended the concept by defining entropy pairs for invariant probability measures. Glasner and Weiss \cite{GW} further developed the notion of entropy tuples, while Huang and Ye \cite{HY} introduced entropy tuples  for invariant  probability measures.
	Concurrently, the study of sequence entropy tuples and tame systems evolved significantly \cite{G1, G2, H1, HLSY1, Q,QY}, alongside the development of a theory of local complexity \cite{BHM, HMY, HY2}.

	Kerr and Li extensively studied independence in \cite{KL3, KH}, uncovering its connection to (sequence) entropy tuples and tameness (i.e., the enveloping semigroup of this system is separable and Fr\'echet \cite{G3}). The concept of independence sets was introduced in \cite{KH} for a tds $(X,G)$ as follows. Given $K\in\N$, for a tuple $\mathcal{A}=(A_1,\cdots, A_K)$ of subsets of $X$,
	we say that a set $J\subset G$ is an independence set for $\A$ if for every nonempty finite subset
	$I\subset J$ and any function $\sigma : I\to  \{1, 2,\cdots, K\}$, we have
	$$\cap_{g\in I}g^{-1}A_{\sigma(g)}\neq\emptyset.$$
	A $K$-tuple $ (x_k)_{k=1}^K\in X^K$ is a $K$-IT-tuple (IT-pair for $K=2$) if 
	for any open
	neighborhood $U_k$ of $x_k$, the tuple $(U_1,U_2,\cdots ,U_K)$ has an infinite independence set. If in addition, $x_i\neq x_j$ for any $i\neq j$, then  $ (x_k)_{k=1}^K$ is called an essential $K$-IT-tuple. Recently, Gómez, León-Torres and  Muñoz-López \cite{GO2025} constructed minimal subshifts that have essential $K$-IT-tuples but do not have  essential $(K+1)$-IT-tuples.
	
	Let  $\pi_{eq}: X\to X_{eq}$ be the factor
	map to the maximal equicontinuous factor of $(X,G)$.  Here the existence of $\pi_{eq}$  can be obtained in	a constructive way via the regionally proximal equivalence relation \cite{Veech}. 	For $K\in\N\cup\{\infty\}$,  $\pi_{eq}$ is said to be almost $K$ to one if $\{y\in X_{eq} :|\pi_{eq}^{-1}(y)|=K\}$ is a residual subset of $X_{eq}$. 	 If $(X,G)$ is minimal, then $(X_{eq},G)$ is also minimal. Consequently, there exists a unique invariant measure $\nu_{eq}$ on $X_{eq}$.
	In this case, $\pi_{eq}$ is said to be  regular $K$ to one if $\nu_{eq}(\{y\in X_{eq}:|\pi_{eq}^{-1}(y)|=K\})=1$.
	
	It is intriguing that the cardinality of the fibers of the maximal equicontinuous factor map $\pi_{{eq}}$ is closely related to the existence of sequence entropy tuples or IT-tuples. Huang, Li, Shao, and Ye~\cite{HLSY1} showed that for a minimal $\mathbb{Z}$-tds admitting no sequence entropy pairs, the factor map $\pi_{{eq}}$ is almost one to one. Garc\'{\i}a-Ramos~\cite{Felipe} further strengthened this result by proving that $\pi_{{eq}}$ is regular one to one.
	For minimal tame (i.e., it has no essential IT-pairs) actions of abelian groups, the almost one to one property of $\pi_{{eq}}$ was independently established by Huang~\cite{H1} and by Kerr and Li~\cite{KH}. A breakthrough toward the non-abelian case was made by Glasner, who proved in~\cite[Corollary 5.4]{G1} that if $(X, G)$ is a minimal tame system  admitting an invariant probability measure, then $\pi_{{eq}}$ is an almost one to one extension. However, he left open the question of whether this extension $\pi_{{eq}}$  is regular one to one \cite[Problem 5.7]{G1}. 
		This question was subsequently answered affirmatively by Fuhrmann, Glasner, J\"ager, and Oertel in \cite[Theorem 1.2]{FG}.
	 In \cite{HLSY}, Huang, Lian, Shao and Ye further developed the method in \cite{FG}, and showed that if a minimal
	system is an almost finite to
	one extension of its maximal equicontinuous factor and has
	no essential $K$-IT-tuples for some $K\ge  2$, then
	it has only finitely many ergodic measures.	In the same paper, they  proposed the following two conjectures  \cite[Conjectures 1 and 2]{HLSY}. 
	\begin{conj}\label{conj1}
		Let $(X,G)$ be a minimal tds that admits no essential $K$-IT-tuples, where $G$ is amenable. Then $\pi_{eq}$ is almost finite to one.
	\end{conj}
	\begin{conj}\label{conj} Let $(X, G)$ be a minimal tds, where $G$ is amenable.  If $\pi_{eq}$ is proximal and not almost one to one, then for each $K\ge 2$, there
		is an essential $K$-IT-tuple.
	\end{conj}

In this paper, we study the dynamics of hyperspace systems and develop a new method for constructing infinite independence sets. As a corollary of our main result (Theorem~\ref{main1'}), we provide an affirmative answer to Conjecture \ref{conj1} in the case where $G$ is virtually nilpotent (Indeed, $\pi_{{eq}}$ is regular finite to one in this case). Meanwhile, we   provide an affirmative answer to Conjecture \ref{conj}.

	\medskip
	
	As another focus of this paper, we investigate sensitivity through localization theory. Motivated by local entropy theory, Ye and Zhang \cite{YZ} introduced the concept of sensitive tuples. Li, Tu, and Ye \cite{LTY} studied sensitivity in its mean form, while more recently, Li, Tu, Ye, and Yu \cite{LY21, LYY22} explored the multi-version of mean sensitivity and its local representation, namely, mean sensitive tuples. Subsequently, Li, Liu, Tu, and Yu \cite{LLTS2024} demonstrated that for a minimal $\Z$-tds, every mean-sensitive tuple is an IT-tuple if $\pi_{eq}$ is almost one to one. They further conjectured that the almost one to one condition is unnecessary. As a corollary of our main results (see Theorem \ref{mainB}), we confirm the validity of their conjecture. In fact, we prove that this conjecture hold for a broader class of minimal systems with actions by amenable groups. 
	If, in addition,  $\pi_{eq}$ is open, we show IT-tuples, IN-tuples,   sequence entropy tuples, sensitive tuples, and weakly mean-sensitive tuples along a F{\o}lner sequence all coincide.
	
	\subsection{Main results}	
	In this subsection, we will state our main results and explain how they address the conjectures and questions mentioned above. In fact, we will also provide additional cases where our results can be applied effectively.
	
	Let us begin with some definitions and notations.  	Let $(X,G)$ be a minimal tds.
	\begin{itemize}
		\item [(a)] $(X,G)$ is said to be incontractible (see Glasner \cite{G1996}) if for every $n\in\N$, $(X^n,G)$ is  Bronstein, i.e., the set of minimal points of $(X^n,G)$ is dense in $X^n$.	
		\medskip
		\item [(b)]$(X, G)$ is said to satisfy the local Bronstein condition, or $(X, G)$ is local Bronstein, (see Auslander \cite{A3}) if, whenever $(x, x')$ is a minimal point in $RP_2(X,G)$, there exist a sequence $\{(x_n, x'_n)\}_{n\in\N}$ of minimal points in $ X \times X$, and  a sequence $\{g_n\}_{n\in\N}$ in $G$, such that
		$$\lim_{n\to\infty}
		(x_n, x'_n)= (x, x') \text{ and } \lim_{n\to\infty}d(g_nx_{n},g_nx'_{n})=0,$$
		where  
		\begin{align*}
			RP_2&(X, G)=\{(x_1,x_2)
			\in X^2 :\forall\epsilon > 0,~\exists x_k'
			\in X\text{ and }g \in G\text{ with } \\
			&d(x_k, x_k')\le \epsilon\ (1\le k\le 2)
			\text{ and }d(gx'_{k_1}, gx_{k_2}')\le \epsilon\ (1\le k_1\le k_2\le 2)\}.
		\end{align*} 
	\end{itemize}

	By a tds $(X,G)$, a nonempty set $A \subset X$ is called an IT-set if every tuple formed by points in $A$ is an IT-tuple.	We now proceed to state our first main result.
	\begin{thm}\label{main1'}Let $(X,G)$ be a minimal tds, and let $\pi_{eq}: X\to X_{eq}$ be the factor
		map to its maximal equicontinuous factor.  Assume  one of the following conditions holds: \begin{itemize}
			\item[(1)] $(X,G)$ is incontractible;
			\medskip
			\item[(2)] $(X, G)$  satisfies the local Bronstein condition, and admits an  invariant  probability measure.
		\end{itemize}
		Then, there exists a dense $G_\delta$ subset $Z\subset X_{eq}$ with $\nu_{eq}(Z)=1$ such that $\pi_{eq}^{-1}(y)$ is an IT-set  for every $y\in Z$.	
	\end{thm}
	\begin{rem}\label{rem-1}
		We note that both conditions in Theorem \ref{main1'}  ensure that the tds $(X,G)$ satisfies the local Bronstein condition. Indeed, it is straightforward to observe that if $(X^2,G)$ is Bronstein, then $(X,G)$  is local Bronstein. 
		
		However, these additional conditions are essential. That is, simply assuming that 
		$(X,G)$  is local Bronstein, is insufficient to establish our theorem.	 
		For example, let $\mathbb{T} = \mathbb{R}/\mathbb{Z}=[0,1)$. Define $T$ and $S$ on $\mathbb{T}$ by $Tx = x + \alpha\mod 1$ and $Sx = x^2$ for all $x \in \mathbb{T}$, where $\alpha$ is an irrational number. Let $G$ be the group generated by $T$ and $S$. Since $(\mathbb{T}, T)$ is minimal, so is $(\mathbb{T}, G)$. The maximal equicontinuous factor of $(\mathbb{T}, G)$ is trivial, and so the fiber of $\pi_{eq}$ is $\mathbb{T}$. By the properties $S$, $\pi_{eq}$ is a proximal extension, i.e.,  $\min_{g\in G} d(gx_{1},gx_{2})=0$ whenever $x_1,x_2\in \mathbb{T}$ with $\pi_{eq}(x_1)=\pi_{eq}(x_2)$. We deduce that the minimal points of $(\mathbb{T} \times \mathbb{T}, G)$ lying in the same fiber must lie in the diagonal $\Delta_\mathbb{T} := \{(x, x) : x \in \mathbb{T}\}$. This implies that $(\mathbb{T}, G)$ satisfies the local Bronstein condition.
		  Additionally, it can be verified that $(\mathbb{T}, G)$ is c-ordered (see \cite[Definition 2.6]{G2}).
		According to \cite[Theorem 2.13]{G2} by Glasner and Megrelishvili, it follows that $(\mathbb{T}, G)$ is null, and hence, it does not have any essential IT-pairs.
	\end{rem}

	Now, we summarize some situations that meet the conditions in Theorem \ref{main1'}. 
	\begin{enumerate}[(a)]

		\item \textbf{$(X,G)$  is a minimal tds, with $G$ as a virtually nilpotent group, or especially an abelian group (Establishing Conjecture \ref{conj1} for these cases).}
		
		A nilpotent group has a lower central series that terminates at the trivial subgroup after finitely many steps. A virtually nilpotent group is one that contains a nilpotent subgroup of finite index. Glasner~\cite{G75} proved that if $G$ is a virtually nilpotent group, then any minimal tds $(X,G)$ is incontractible. Since every virtually nilpotent group is amenable, it always admits an invariant probability measure. Hence, by Remark~\ref{rem-1}, such systems also satisfy condition~(2) in Theorem~\ref{main1'}.

		By applying Theorem~\ref{main1'} to these settings, we provide an affirmative answer to~\cite[Question 4.8]{MaassShao}, which asks whether every minimal $\mathbb{Z}$-system with finite maximal sequence entropy is of finite type.
		
		\medskip
		
		\item  \textbf{$(X,G)$ is a  point distal tds.}
		
		Following Veech \cite{Vee}, a tds $(X,G)$ is said to be a point distal tds, if $(X,G)$ contains a distal point $x\in X$ whose orbit is dense in $X$. Here a point $x\in X$ is called a distal point if $\min_{g\in G}d(gx,gx')>0$ for any $x'\in X$ with $x'\neq x$. Denote by $D(X,G)$ the set of distal points of $(X,G)$.
		Now we prove that each point distal tds $(X,G)$ is incontractible. 
		
		Let $(X,G)$ be a point distal tds.  It follows from Auslander-Ellis Theorem that any distal point is a minimal point, and hence a point distal tds is  minimal. Since the nonempty set $D(X,G)$ is $G$-invariant and $(X,G)$ is minimal, one has  $D(X,G)$ is dense in $X$. By the definition of distal points, it is easy to check that $D(X,G)^n\subset D(X^n,G)$, which implies that $D(X^n,G)$ is dense in $X^n$ for each $n\in\N$. Using the fact that 
		each distal point is a minimal point again, we have the set of minimal points of $(X^n,G)$ is dense in $X^n$ for each $n\in\N$, and hence $(X,G)$ is incontractible.
		
		\medskip

		\item\textbf{${(X,G)}$   is a minimal and almost finite to one extension of its maximal equicontinuous factor.}
		
		It follows from \cite[Theorem 3.16]{CS} that, if  a minimal tds is an almost finite to one extension of its maximal equicontinuous factor, then it is a point distal tds. Thus, this case is a special case of situation (b). Huang, Lian, Shao and Ye \cite{HLSY} under the assumption that  $\pi_{eq}$ is almost $N$ to one, and there is some
		integer $K\ge 2$ such that $(X, G)$ has no essential $K$-IT-tuples,  proved that $(X,G)$ has no more than $N(K-1)^N$ ergodic measures.
		Recently, the upper bound was improved to $N(K-1)$ by Wang and the first two authors of this paper \cite{LWX}. By applying Theorem \ref{main1'} in this situation, that is, $\pi_{eq}$ is almost finite to one, we improve the upper bound to $(K-1)$ (see Corollary \ref{cor-B1} (iv)). This resolves the optimal upper bound aimed by Huang et al. \cite[p.5, Remark (2)]{HLSY}.
		
		Now, we explain why this upper bound of $(K-1)$ appears to be optimal. As shown in \cite[Subsection 5.3]{FG}, there exists a minimal system, which is an at
		most two to one extension of its maximal equicontinuous
		factor, that has no essential $3$-IT-tuples and exhibits two distinct ergodic probability measures.

		\medskip
		
		\item \textbf{$(X,G)$ is a minimal, proximal extension  of its  maximal equicontinuous factor, and admits  an invariant  probability measure (Establishing Conjecture \ref{conj}).}
		
		Recall that $\pi_{eq}$ is called a proximal extension if  $\min_{g\in G} d(gx_{1},gx_{2})=0$ whenever $x_1,x_2\in X$ with $\pi_{eq}(x_1)=\pi_{eq}(x_2)$.  
		If 	a minimal point  of $(X\times X,G)$ belongs to $RP_2(X,G)$ then $\pi_{eq}(x_1)=\pi_{eq}(x_2)$, implying  $\min_{g\in G} d(gx_{1},gx_{2})=0$. In particular,   $(X,G)$ is local Bronstein.

		Now we show Theorem \ref{main1'} implies that Conjecture \ref{conj} holds.
		In fact,  since $\pi_{eq}$ is proximal and not almost one to one, it follows from \cite[Corollary 2.17]{HLSY} that $\pi_{eq}$ is infinite to one. Thus, Theorem \ref{main1'} shows that for $\nu_{eq}$-a.e. $y\in X_{eq}$, $\pi_{eq}^{-1}(y)$ contains arbitrarily large essential IT-tuples, and hence Conjecture \ref{conj} holds.
		
		Notably,  by Corollary \ref{cor-B1} (iii), we prove a stronger result that,  for every $y\in X_{eq}$, $\pi_{eq}^{-1}(y)$ contains arbitrarily large essential IT-tuples.
	\end{enumerate}

	\medskip
	\medskip

We now turn to the second topic of this paper: mean sensitivity. To better illustrate the main results, we begin by introducing the notion of a weakly mean-sensitive tuple.
	\begin{defn}\label{defn-mean-sensitive}Let $(X,G)$ be a tds, where  $G$ is amenable.
		A tuple $(x_k)_{k=1}^K\in X^K$ is a   weakly mean-sensitive $K$-tuple (weakly mean-sensitive pair for $K=2$) if 
		for any open
		neighborhood $U_k$ of $x_k$, $k=1,2,\cdots,K$, there exists $\delta>0$ such that for any nonempty open subset $U$ of $X$, $$\overline{BD}(\{g\in G: (gU)\cap U_k\neq\emptyset\text{ for }1\le k\le K\})>\delta,$$
		where $\overline{BD}(A)$ is the upper Banach density of $A$ (see Section \ref{s-2} for definition). 
	\end{defn}

	A set is called a  weakly mean-sensitive set  if every tuple composed of  points within it is a  weakly mean-sensitive tuple.  Corresponding to Theorem \ref{main1'} for IT-sets, we present the following results.
	\begin{thm}\label{mainB}Let $(X,G)$ be a minimal tds, where  $G$ is amenable. If  $(X,G)$ satisfies the local Bronstein condition, then
		the following  statements hold:
		\begin{enumerate}[(i)]
			\item There exists  a dense $G_\delta$ subset $Z\subset X_{eq}$ with $\nu_{eq}(Z)=1$ such that  $\pi_{eq}^{-1}(y)$ is a  weakly mean-sensitive set for every $y\in Z$.
			\item Assume $K\in\N$ and $\pi_{eq}$ is regular $K$ to one. Then, $(X,G)$ has essential  weakly mean-sensitive $K$-tuples but  has no essential  weakly mean-sensitive $(K+1)$-tuples. 
			
			\item Every  weakly mean-sensitive set is  an IT-set.	
		\end{enumerate}
	\end{thm}
	\begin{rem} Notice that the converse of Theorem \ref{mainB} (iii) does not necessarily hold; specifically, there exist a minimal tds whose essential IT-pairs  are not essential weakly mean-sensitive pairs. For example, Fuhrmann and Kwietniak \cite[Section 3]{FuhrmannKwietniak2020} constructed a class of minimal systems that are regular one to one extensions of its maximal equicontinuous factor, and possess essential IT-pairs. By Theorem \ref{mainB} (ii), this class of systems has no essential weakly mean-sensitive pairs. Therefore, for a minimal tds, an IT-pair may not always be a weakly mean-sensitive pair.
	\end{rem}

The notion of mean sensitive tuples for $\mathbb{Z}$-actions was introduced by Li, Ye, and Yu~\cite{LYY22}, and can be naturally extended to amenable group actions as follows. Let $(X, G)$ be a tds with $G$ amenable, and let $\mathcal{F}$ be a F{\o}lner sequence in $G$.
	A $K$-tuple $ (x_k)_{k=1}^K\in X^K$ is a  mean-sensitive $K$-tuple along $\mathcal{F}$  if 
	for any open
	neighborhood $U_k$ of $x_k$ ($k=1,2,\cdots,K$), there exists $\delta>0$ such that for any nonempty open subset $U$ of $X$, there exist $y_1,\cdots,y_K\in U$ such that $$\limsup_{n\to\infty}\frac{1}{|F_n|}|\{g\in F_n:gy_k\in U_k\text{ for }1\le k\le K\}|>\delta,$$
	where $|A|$ is the number of elements in a finite set $A$.  It is straightforward to see that every mean-sensitive tuple is also a weakly mean-sensitive tuple.
Li, Liu, Tu, and Yu \cite{LLTS2024} demonstrated that for a minimal tds $(X, \mathbb{Z})$ with $\pi_{{eq}}$ being almost one to one, every mean-sensitive $K$-tuple along the F{\o}lner sequence $\{ [0, n-1] \cap \mathbb{Z} \}_{n \in \mathbb{N}}$ is a $K$-IT-tuple. They further conjectured that the almost one to one condition on $\pi_{\mathrm{eq}}$ is superfluous.
In particular, Theorem \ref{mainB} allow us to confirm their conjecture in a broader setting and derive a more general result.
	\begin{thm}\label{thm123}
		Let $(X,G)$ be a minimal tds, where  $G$ is amenable, and $\mathcal{F}$ be a F{\o}lner sequence of $G$. If $(X,G)$ is local Bronstein, then each mean-sensitive $K$-tuple along $\mathcal{F}$ is an $K$-IT-tuple for $K\in\N$ with $K\ge2$.
	\end{thm}

	A factor map $\pi: (X, G) \to (Y, G)$ between two minimal tds is open if it maps open subsets of $X$ to open subsets of $Y$. Every distal extension is open (see \cite{G2005}), and any factor map between minimal systems can be lifted to an open extension via almost one to one extensions \cite[Chapter VI, Section 3]{Vr}.
	
As in \cite{HLY,YZ}, for $K\ge 2$, a tuple  $ (x_1,\cdots,x_K)$ of $X$ is a sensitive tuple (a sensitive pair if $K=2$) if for open
neighborhoods $U_i$ of $x_i$, $i=1,2,\ldots,K$, and any nonempty open subset $U$ of $X$, there
exists $g\in G$ such that $U\cap (gU_k)\neq\emptyset$ for $k=1, 2,\cdots,K$. Now, under the assumptions of Theorem~\ref{thm123} and assuming that $\pi_{eq}$ is open, we derive the following result. It follows directly from Corollary \ref{cor-B1} (ii), Corollary \ref{cor-S1} (ii), and the fact that each IT-tuple is a sequence entropy tuple, which coincides with an IN-tuple. \cite[Theorem 5.9]{KH}.
	\begin{thm}	Let $(X,G)$ be a minimal tds, where  $G$ is amenable, and $\mathcal{F}$ be a F{\o}lner sequence of $G$. 
		If  $\pi_{eq}$ is open and $(X,G)$ is local Bronstein,  then essential  IT-tuples, essential IN-tuples,  essential  sequence entropy tuples, essential  sensitive tuples, and essential  weakly mean-sensitive tuples along $\mathcal{F}$  all coincide.  
	\end{thm}

	\medskip
	
	\noindent	{\bf Structure of this paper:} In Section \ref{s-2}, we review basic concepts and prove some lemmas. In Section \ref{sec-isp}, we  find the relation between a minimal tds  and its hyperspace system, identifying a class of closed sets that satisfy the interior saturation property. In Section \ref{IT-ab}, we prove Theorem \ref{main1'}, while Section \ref{s-mean-s-set} focuses on the proof of Theorem \ref{mainB}.

	\section{Preliminaries}\label{s-2}
	Throughout this paper, we denote by  $\N$ and $\Z$ the sets of  natural numbers and integers, respectively. Let $X$ be a compact metric space with a metric $d$.  We denote the interior, closure and boundary of a set $A$ in $X$ by $\text{int}(A)$, $\overline{A}$ and $\partial A$, respectively. 	For $x\in X$, a nonempty subset $V\subset X$ and $\epsilon>0$, denote $$B^X_\epsilon(x)=\{x'\in X:d(x,x')<\epsilon\}\text{ and }B^X_\epsilon(V)=\{x'\in X:d(V,x')<\epsilon\},$$
	where $d(V,x')=\inf_{z\in V}d(z,x')$.  We denote the diameter of a set $A$ by $\text{diam}(A)$. 
	\subsection{Subset of group}  Let $G$ be a countable group. A set $\mathcal{S}\subset G$ is said to be a syndetic set of $G$ if there exists a finite set $K \subset G$ such that
	$Kg\cap \mathcal{S}\neq\emptyset \text{ for all }g\in G.$ A set $\mathcal{S}\subset G$ is said to be a thick set of $G$ if for any finite set $K \subset G$, there exists $g \in G$ such that $Kg \subset \mathcal{S}$. A set $\mathcal{S}\subset G$ is said to be a thickly syndetic set of $G$ if for any finite set $K \subset G$, the set $\{g \in G : Kg \subset \mathcal{S}\}$ is a syndetic set. Denote $\mathcal{F}_s(G)$, $\mathcal{F}_t(G)$, and $\mathcal{F}_{ts}(G)$ as the sets of syndetic, thick, and thickly syndetic sets of $G$, respectively. The following lemma is from \cite[Proposition 3.2]{W} and \cite[Theorem 2.4]{ber}.
	\begin{lem}\label{F} Let $G$ be a  countable  group. Then, the following statements hold.
		\begin{itemize}
			\item[(1)] If $\mathcal{S}_1,\mathcal{S}_2\in  \mathcal{F}_{ts}(G)$, then $\mathcal{S}_1\cap\mathcal{S}_2\in \mathcal{F}_{ts}(G)$.
			\item[(2)]  If $\mathcal{S}_1\in \mathcal{F}_{s}(G)$ and $\mathcal{S}_2\in  \mathcal{F}_{t}(G)$, then $\mathcal{S}_1\cap\mathcal{S}_2\neq\emptyset$.
			\item[(3)]  If $\mathcal{S}_1\in \mathcal{F}_{s}(G)$ and $\mathcal{S}_2\in  \mathcal{F}_{ts}(G)$, then  $\mathcal{S}_1\cap\mathcal{S}_2\in \mathcal{F}_{s}(G)$.
		\end{itemize}
	\end{lem}

	Recall that a countable group $G$ is amenable (see \cite{F}) if and only if there exists a sequence of finite, nonempty subsets $\mathcal{F} = \{F_n\}_{n=1}^\infty$ of $G$, called a F\o lner sequence, such that $$\lim_{n \to \infty} \frac {|gF_n\Delta F_n|}{|F_n|} = 0\text{ for every $g\in G$.}$$  Given a subset $F$ of $G$, the upper
	and lower
	Banach density of $F$ are defined, respectively, by
	$$\overline{BD}(F)=\max_{\mathcal{F}}\limsup_{n\to\infty}\frac{|F_n\cap F|}{|F_n|}$$
	and
	$$\underline{BD}(F)=\min_{\mathcal{F}}\liminf_{n\to\infty}\frac{|F_n\cap F|}{|F_n|},$$
	where $\mathcal{F}=\{F_n\}_{n=1}^\infty$ takes over all F{\o}lner sequences of $G$. It is clear that $\underline{BD}(F)\le \overline{BD}(F)$ for any $F\subset G$. When a set
	is denoted with braces, for example $\{A\}$, we simply write $\underline{BD}\{A\} =\underline{BD}(\{A\})$ and $\overline{BD}\{A\} =\overline{BD}(\{A\})$.

	\subsection{Topological dynamical system}\label{sb-2}	 A tds $(X,G)$ is called minimal if $X$ 
	contains no proper nonempty closed invariant
	subsets.  We denote the orbit of $x\in X$ by
	$Orb(x, G):=\{gx:g\in G\}.$
	A point $x\in X$ is called minimal if the system $(\overline{Orb(x, G)},G)$ is minimal. Denote by $Min(X,G)$ the set of all minimal points of $(X,G)$, which is nonempty.  The closure of $Min(X,G)$ in $X$ is called the minimal center of $(X,G)$.
	A subset $M$ of $X$ is a minimal set if $M=\overline{Orb(x, G)}$ for some minimal point $x$ of $(X,G)$. 	 For a subset $U$ of $X$ and $x\in X$, we denote
	\[N(x,U)=\{g\in G:gx\in U\}.\]
	The following results are well-known (see e.g.  \cite{Furbook}).
	\begin{lem}\label{syndetic}Let $(X,G)$ be a tds. Then, the following statements hold:
		\begin{itemize}
			\item[(1)] If $x\in X$ is a minimal point, then for any open neighborhood $V$ of $x$, $N(x,V)\in\mathcal{F}_s(G)$.
			\item[(2)]If $\mu$ is an invariant probability measure of $(X,G)$, then for any measurable set $V\subset X$ with $\mu(V)>0$,  $\{g\in G: \mu(gV\cap V)>0\}\in\mathcal{F}_s(G)$.
			\item[(3)] Let $x\in X$, and $U\subset X$ be an open neighborhood  of $\overline{Min(X,G)}$. Then, $N(x,U)\in \mathcal{F}_{ts}(G)$.
			\item[(4)]Let $x\in X$, and let $M\subset \overline{Orb(x,G)}$ be a minimal set. Then, for any open neighborhood $U$ of $M$, $N(x,U)\in \mathcal{F}_{t}(G)$.
		\end{itemize}
	\end{lem}

	A factor map $\pi: X\to Y$ between two tdss $(X,G)$ and $(Y,G)$ is a continuous onto map which intertwines the actions. In this case, we say that $(Y, G)$ is a factor of $(X, G)$, and $(X, G)$ is an extension of $(Y, G)$. 	The following results are well-known (see e.g. \cite[p. 21, Proposition 14]{A1}).
	\begin{lem} Let $\pi: (X,G)\to (Y, G)$ be a factor map. Then the following statements hold:
		\begin{itemize}
			\item[(1)] $\pi(Min(X,G))=Min(Y,G)$; 
			\item[(2)] If $(X,G)$ is minimal, then $\pi$ is semi-open, i.e., for any nonempty open subset $U$ of $X$, $\pi(U)$ has nonempty interior.
		\end{itemize} 
	\end{lem}

	For $K\in\N\cup\{\infty\}$, a factor map $\pi: (X,G)\to (Y,G)$ is almost $K$ to one if $\{y\in Y :|\pi^{-1}(y)|=K\}$ is a residual subset\footnote{A residual set is a subset that contains a countable intersection of dense open sets.} of $Y$. In particular, $\pi$ is said to be almost finite to one if this property holds for  some $K\in\N$. It follows from  \cite[Proposition 2.15]{HLSY} that if $(Y,G)$ is a minimal tds
 then  $\pi: (X,G)\to (Y,G)$ is almost $K$ to one, where 	$K=\min_{y\in Y}|\pi^{-1}(y)|\in\N\cup\{\infty\}$.

	The notion of $K$-regionally proximal relation, introduced in \cite{A} is defined for $K\ge 2$ by
	\begin{align*}
		RP_K&(X, G)=\{(x_i)_{i=1}^K
		\in X^K :\text{ for any }\epsilon > 0\text{ there exist }x_k'
		\in X\text{ and }g \in G\\
		&\text{ with } d(x_k, x_k')\le \epsilon  (1 \le  k \le K),
		\text{ and }d(gx'_{k_1}, gx_{k_2}')\le \epsilon (1\le k_1\le k_2\le K)\}.
	\end{align*}
	
	\begin{lem}\label{lem-rp}Let $(X,G)$ be a minimal tds, and  $(x_1, x_2,\cdots , x_K)\in RP_K(X,G)$. Then, for any $\epsilon>0$ and nonempty open set $U\subset X$, there exist $h\in G$ and nonempty open subsets $V_k\subset B_\epsilon^X(x_k)$ for $k=1,2,\cdots,K$, such that 
		$$hV_k\subset U,\text{ for all }k=1,2,\cdots,K.$$
	\end{lem}
	\begin{proof}Since $(X,G)$ is minimal,  $\{gU:g\in G\}$ is an open cover of $X$. Let $\gamma>0$ be a Lebesgue number of $\{gU:g\in G\}$. Since   $(x_1, x_2,\cdots , x_K)\in RP_K(X,G)$, there exist $h'\in G$ and $ x'_k\in B_\epsilon^X(x_k)$ for $k=1,2,\cdots,K$,  such that $$d(h'x'_k,h'x'_{k'})<\gamma,1\le k\le k'\le K.$$
		Thus, there is $g\in G$ such that $$h'x'_k\in gU\text{ for all }1\le k\le K.$$
		This implies $g^{-1}h'x'_k\in U$ for $1\le k\le K$. Let  $h=g^{-1}h'$. Then by the continuity of $h$, there exist nonempty open neighborhood $V_k$ of $x_k'$ with $V_k\subset B_\epsilon^X(x_k)$ for $k=1,2,\cdots,K$, such that 
		$hV_k\subset U,\text{ for all }k=1,2,\cdots,K.$
	\end{proof}
	
	A tds $(X,G)$ is equicontinuous if for any $\epsilon > 0$, there is $\delta > 0$ such that whenever $x, y \in X$ with $d(x, y) < \delta$, then $d(gx,gy) < \epsilon$ for all $g \in G$. There is a smallest invariant equivalence relation $S_{eq}$ such that the quotient system $(X/S_{eq},G)$ is equicontinuous \cite{EG}. The equivalence relation $S_{eq}$ is called the equicontinuous structure relation, and the factor $(X/S_{eq},G)$ is called the maximal equicontinuous factor of $(X,G)$. 	In this paper, we denote  $X_{eq}=X/S_{eq}$  and  $\pi_{eq}:(X,G)\to(X_{eq},G)$ is the corresponding factor map.  It is well known that
	$(X_{eq}, G)$ is both minimal and uniquely ergodic. The ergodic measure is denoted by $\nu_{eq}$, and the metric on $X_{eq}$ is denoted by $d_{eq}$. Notice that  using an equivalent
	metric we always assume that $(X_{eq},G)$ is an isometry, that is, $d_{eq}(y,y')=d_{eq}(gy,gy')$ for $y,y'\in X_{eq}, g\in G$.  
	
	We remark that $S_{eq}$ is the smallest closed $G$-invariant equivalence relation containing the $2$-regionally proximal relation $RP_2(X,G)$.  It is easy to see that $RP_2(X,G)$ is an equivalence relation if and only if $S_{eq} = RP_2(X,G)$. In many cases, $RP_2(X,G)$ is an equivalence relation. For example, $(X,G)$ admits an invariant probability  measure \cite{Mc}, or $(X,G)$ is local Bronstein \cite{A}.
	\begin{thm}[\text{\cite[Theorem 8]{A}}]\label{rp}Let $(X,G)$ be a minimal tds, and meets  the local Bronstein condition. For $K\ge 2$, if $x_1, x_2,\cdots, x_K \in X$ satisfy $(x_1, x_k)\in RP_2(X,G)$
		for $2\le k\le K$, then $(x_1, x_2,\cdots , x_K)\in RP_K(X,G)$.
	\end{thm}
	\begin{rem}
		The original condition in \cite[Theorem 8]{A} is a group-theoretic condition on the Ellis group of the system, which is weaker than the condition stated in Theorem \ref{rp}; see \cite[Remark 7.2]{LY} for a proof of Theorem \ref{rp}.
	\end{rem}
	\begin{rem}If the local Bronstein condition in Theorem \ref{rp} can be replaced by the assumption that $(X, G)$ admits an invariant  probability measure, then the requirement that $(X, G)$ is local Bronstein in Theorem \ref{main1'} (2) becomes redundant. Notably, if this replacement holds, it would confirm the Conjecture \ref{conj1}.
		\end{rem}

	\subsection{System on the Hyperspace}\label{hspace}Let $X$ be a compact metric space with metric $d$, and $2^X$ be the set of nonempty
	closed subsets of $X$ endowed with the Hausdorff metric $d_H$, which can be  defined by
	$$d_H(A, B)
	= \max\{\max_{a\in A}
	d(a, B),\max_{b\in B}
	d(b, A)\},$$
	where $d(x, A) = \inf_{y\in A} d(x, y)$ for $x\in X$ and $A\subset X$. Then $(2^X,d_H)$ is a compact metric space, as $X$ is a compact  metric space. 
	We say a sequence  $\{A_i\}_{i=1}^\infty$ in $2^X$
	converges to $A\in 2^X$, denoted by $\lim_{i\to\infty} A_i = A$, if
	$$\lim_{i\to\infty}d_H(A_i,A)=0.$$
	The following results are well known (see e.g. \cite{Ku,Ku2}).
	\begin{lem}\label{usc} Let $X,Y$ be compact metric spaces and $\pi:X\to Y$ be a continuous surjection. Then we have the following statements:
		\begin{enumerate}
			\item $\pi^{-1}:Y\to 2^X, y\mapsto \pi^{-1}(y)$ is upper semi-continuous, and hence the continuity points of $\pi^{-1}$ form a dense $G_\delta$ subset of $Y$;
			\item $\pi^{-1}$  is continuous if and only if $\pi$ is open.
		\end{enumerate} 
	\end{lem}
	It is well known that each tds $(X,G)$ induces a tds on  $2^X$, known as the hyperspace system. More precisely, for $g\in G$, the action on $2^X$ is given	by $gA = \{gx : x\in A\}$ for each $A\in 2^X$.
	
	The following lemma will be used in our proof.
	\begin{lem}\label{lem-0}	Let $(X,G)$ be a minimal tds and $\pi_{{eq}}$ is the factor map to its maximal equicontinuous factor  $(X_{eq},G)$. Then,  the following statements hold: \begin{itemize}
			\item[(1)] If $C\in 2^{X_{eq}}$ and $C_*\in\overline{Orb(C,G)}\subset 2^{X_{eq}}$, then $\text{diam}(C_*)=\text{diam}(C)$.
			\item[(2)] Let $C_1,\cdots,C_M\in 2^{X}$ and $\{g_n\}_{n=1}^\infty$ be a sequence of $G$. If  $\lim_{n\to\infty}g_nC_m=C_m'$ for each $m\in\{1,2,\cdots,M\}$ and  $\lim_{n\to\infty}g_n(\cup_{m=1}^MC_m)=C'$, then
			$$C'=\cup_{m=1}^MC_m'.$$
		\end{itemize}
	\end{lem} 
	\begin{proof}(1) is straightforward  from the assumption that $(X_{eq}, G)$ is an isometry.
		
		For	(2), we define a map $\tau: (2^X)^M\to 2^X$ by $\tau(A_1,\cdots,A_M)=\cup_{m=1}^MA_m$. By the definition of Hausdorff metric, we have for any $(A_1,\cdots,A_M),(A_1',\cdots,A_M')\in (2^X)^M$, 
		\[d_H(\cup_{m=1}^MA_m,\cup_{m=1}^MA_m')\le \max_{1\le m\le M}d_H(A_m,A_m').\]
		which implies $\tau$ is continuous. Thus,
		\begin{align*}
			C'=	\lim_{n\to\infty}g_n  (\cup_{m=1}^MC_m)&=\lim_{n\to\infty}g_n\tau(C_1,\cdots,C_M)\\
			&=\tau(\lim_{n\to\infty}g_n(C_1),\cdots,\lim_{n\to\infty}g_n(C_M))\\
			&=\tau(C_1',\cdots,C_m')=\cup_{m=1}^NC_m'.
		\end{align*}
		The proof is complete.
	\end{proof}

	\subsection{Ellis semigroup}
	Given a tds $(X,G)$, the Ellis semigroup $E(X)(=E(X,G))$ associated to
	$(X, G)$ is defined as the closure of $\{x\mapsto  tx: t \in G \}\subset X^X$ in the product topology, where
	the semi-group operation is given by the composition. On $E(X)$, we may consider the
	tds given by $E(X)\ni s\mapsto  ts$ for each element $t\in G$.
	
	We now summarize some properties of the tds $(E(X), G)$, as detailed in \cite[pp. 52-53]{A1}, \cite[Theorems 5.4 and 6.20]{P} and \cite[Theorem 15.13]{Haar}.
	\begin{thm}\label{thm-A1}Suppose $H$ is a compact metric space and $(H, G)$ is minimal
		and equicontinuous. Then we have the following statements:
		\begin{enumerate}[(a)]
			\item $E(H)$ is a compact metrizable topological group, and the system $(E(H), G)$ is uniquely ergodic with respect to its Haar measure.
			\medskip
			\item $G$ as a group of homeomorphisms on $H$, can be seen as a dense subgroup of $E(H)$. Hence, $(E(H),G)$ is minimal.
				\medskip
			\item $(H, G)$ is a factor of $(E(H), G)$, with an open factor map $\pi_H$ is given by
			$$\pi_H: E(H)\to  H, t\mapsto  th, \text{ for some fixed $h\in H$}.$$
		\end{enumerate}
	\end{thm}

	\section{Interior saturation property}\label{sec-isp}
	The aim of this section is to establish the relation between a minimal tds satisfying the conditions in Theorem \ref{main1'} and its hyperspace system (see Proposition \ref{prop-2}). However, for our purposes, we only need a corollary of this result, namely Lemma \ref{lem-2}, which is applied in Section \ref{IT-ab} and Section \ref{s-mean-s-set} to prove our main results.

	\subsection{Notation}\label{subsection-notain}	In this subsection, we introduce some notations, and prove some basic properties about them.

	Let $(X,G)$ be a minimal tds, and let $\pi_{eq}:(X,G)\to(X_{eq},G)$ be the factor map to the maximal equicontinuous factor of $(X,G)$. Recall that $(X_{eq}, G)$ is an isometry and that the unique invariant measure on $X_{eq}$ is denoted by $\nu_{eq}$.

	Let $$\mathcal{X}=\overline{\{\pi_{eq}^{-1}(y): y\in X_{eq}\}}\subset 2^X.$$ 
Then 	$\mathcal{X}$ is a compact invariant subset of $2^X$, and   thus, $(\X,G)$  is a subsystem of $(2^X,G)$. Note that  for every $E\in\mathcal{X}$, there exists $y\in X_{eq}$ such that $E\subset \pi_{eq}^{-1}(y)$. Therefore, we	may define the map $$\pi_\mathcal{X}: \mathcal{X}\to X_{eq},~E\mapsto\pi_{eq}(E).$$

	We have the following description for the elements in $\X$.
	\begin{lem}\label{lem-1.2}
		Let $(X,G)$ be a minimal tds. If $\pi_{eq}$ is almost $N$ to one  for some $N\in\N\cup\{\infty\}$, then  $|E|\ge N$ for all $E\in\X$.
	\end{lem}
		\begin{proof}
			Suppose,  for a contradiction,  that there exists $E_*\in \mathcal{X}$ with $|E_*|=M<N$.  Our aim is to show that $\{y\in X_{eq}:|\pi_{eq}^{-1}(y)|\le M\}$ is a residual subset in $X_{eq}$, contradicting the assumption that $\pi_{eq}$ is almost $N$ to one. This completes the proof of Lemma \ref{lem-1.2}.
			
			For $\gamma>0$ and $A\subset X$, denote 
			$$s(A,\gamma)=\min\{m\in\N: \text{there exist }x_1,\cdots,x_m\in X\text{ such that }A\subset \bigcup_{i=1}^mB^X_\gamma(x_i)\}.$$ 
			Fix $\epsilon>0$ and $\delta>0$. Since $(X_{eq}, G)$ is a minimal tds, there exists a finite subset $F$ of $G$ such that \begin{align}\label{eq-00}\{gy: g\in F\}\text{ is }\delta\text{-dense in }X_{eq}\text{ for all }y\in X_{eq}.\end{align} 
			There is $\xi>0$ such that for any
			$x, x'\in X$, $d(x,x')<\xi\text{ implies }d(gx,gx')<\epsilon\text{ for all }g \in F.$
			
			Since $|E_*|=M$, there exists $y_*\in X_{eq}$ such that $ s(\pi_{eq}^{-1}(y_*),\xi)\le M$. Then, $  s(\pi_{eq}^{-1}(gy_*),\epsilon)\le M$ for all $g\in F.$ Thus, by \eqref{eq-00},  $Y_\epsilon=\{y\in X_{eq}:s(\pi^{-1}_{eq}(y),\epsilon)\le M\}$ is $\delta$-dense in $X_{eq}$. Noting that this holds for all $\delta>0$, one has 
			$Y_\epsilon$ is dense in $X_{eq}$. Since $\pi_{eq}^{-1}$ is upper semi-continuous, $Y_\epsilon$ is  an open subset of $X_{eq}$. Therefore, by the arbitrariness of $\epsilon>0$, $Y_*=\cap_{\epsilon>0}Y_\epsilon$ is a dense $G_\delta$ subset of $X_{eq}$. It is clear that 
			$$|\pi_{eq}^{-1}(y)|\le M<N\text{ for every }y\in Y_*.$$
			This implies $\pi_{eq}$ is  almost $M$ to one, which is  a contradiction.  Hence $|E|\ge N$ for all $E\in\mathcal{X}$.
		\end{proof}

	Now, we construct a subsystem $(\mathcal{X}_{eq}^{\text{meas}},G)$  of $(\X,G)$ as follows. Put
	$$\mathcal{X}_{eq}^{\text{meas}}=\{E\in\mathcal{X}:\nu_{eq}(\pi_\mathcal{X}(B^\mathcal{X}_{\epsilon}(E)))>0, \forall \epsilon>0\}.$$
	It is clear that $\mathcal{X}_{eq}^{\text{meas}}$ is a closed and $G$-invariant subset of $\X$.  The following lemma shows that $\mathcal{X}_{eq}^{\text{meas}}$ contains full measure fibers of $X_{eq}$. In particular, it is nonempty, and hence $(\mathcal{X}_{eq}^{\text{meas}},G)$ is a subsystem of $(\X,G)$. 
	\begin{lem}\label{lem-1}Let $(X,G)$ be a minimal tds and denote
		$$Y=\{y\in X_{eq}: \pi^{-1}_{eq}(y)\in\mathcal{X}_{eq}^{\text{meas}}\}.$$ Then, $Y$ is a dense $G_\delta$ subset of $X_{eq}$ with $\nu_{eq}(Y)=1$.
	\end{lem}
	\begin{proof}
First, we show that $Y$ is a dense $G_\delta$ subset of $X_{eq}$. It is easy to see that every continuity point of $\pi_{eq}^{-1}$ belongs to $Y$. By Lemma~\ref{usc}, the set of continuity points of $\pi_{eq}^{-1}$ is dense in $X_{eq}$. Hence, $Y$ is a dense subset of $X_{eq}$.
 Now we prove $Y$ is a $G_\delta$ subset.
		For $\epsilon>0$, denote  $$Y_\epsilon=\{y\in X_{eq}: \pi_{eq}^{-1}(y)\subset B^X_\epsilon(E)\text{ for some }E\in  \mathcal{X}_{eq}^{\text{meas}}\}.$$
	By the upper semi-continuity of $\pi_{eq}$, one has  $Y_\epsilon$ is open for all $\epsilon>0$. Thus, it suffices to prove that  $Y=\cap_{n\in\N}Y_{1/n}$. As  $Y\subset\cap_{n\in\N}Y_{1/n}$ is trivial, we only need to show that $\cap_{n\in\N}Y_{1/n}\subset Y$. 
	
	Let $y_*\in \cap_{n\in\N}Y_{1/n}$. There exists $E_n\in \mathcal{X}_{eq}^{\text{meas}}$ such that $\pi_{eq}^{-1}(y_*)\subset B_{1/n}^X(E_n)$ for all $n\in\N$.  Passing to a subsequence if necessary, we may assume that $\lim_{n\to\infty}E_n=E_*\in \mathcal{X}_{eq}^{\text{meas}}$. Then $\pi_{eq}^{-1}(y_*) \subset E_*$, which implies that $\pi_{eq}^{-1}(y_*) = E_* \in \mathcal{X}_{eq}^{\text{meas}}$, and thus $y_* \in Y$. By the arbitrariness of $y_* \in \cap_{n \in \mathbb{N}} Y_{1/n}$, it follows that $\cap_{n \in \mathbb{N}} Y_{1/n} \subset Y$. Therefore, $Y$ is a $G_\delta$ subset of $X$.

		On the other hand, for $\epsilon>0$, by Lusin's theorem, we may find a measurable subset $Z_\epsilon$ of $X_{eq}$ with $\nu_{eq}(Z_\epsilon)>1-\epsilon$ such that $\pi_{eq}^{-1}|_{Z_\epsilon}$ is continuous, and hence the support of $\nu_{eq}|_{Z_\epsilon}$ belongs to $Y$. We finish the proof of Lemma \ref{lem-1}, as $\nu_{eq}( \cup_{n=1}^\infty Z_{1/n})=1$.
	\end{proof}
	We now enumerate additional properties of 
	$\mathcal{X}_{eq}^{\text{meas}}$
	to facilitate its use in our subsequent work. \begin{lem}\label{lem-per-xmeas}
		Let $(X,G)$ be a minimal tds. Then the following statements hold:
		\begin{itemize}
			\item[(1)]If $\pi_{eq}$ is open, then $\mathcal{X}_{eq}^{\text{meas}}=\{\pi_{eq}^{-1}(y): y\in X_{eq} \}$.
			\item[(2)] If $\pi_{eq}$ is regular one to one, then $|E|=1$ for every $E\in \mathcal{X}_{eq}^{\text{meas}}$.
			\item[(3)] If $\pi_{eq}$ is regular $K$ to one, then $|E|\le K$ for every $E\in \mathcal{X}_{eq}^{\text{meas}}$.
		\end{itemize}
	\end{lem}
	\begin{proof}It is easy to see that (1) is from Lemma \ref{usc} (2), and  (2) is from (3). Now we prove (3), suppose for a contradiction that there exists $E\in \mathcal{X}_{eq}^{\text{meas}}$ with $|E|\ge K+1$. By the definition of Hausdorff metric, there exists $\epsilon>0$ such that for any $E'\in B^\mathcal{X}_{\epsilon}(E)$, $|E'|\ge K+1$. Thus, for any $y\in\pi_\mathcal{X}(B^\mathcal{X}_{\epsilon}(E))$, $|\pi^{-1}(y)|\ge K+1$. However, as $E\in \mathcal{X}_{eq}^{\text{meas}}$, it follows that $\nu_{eq}(\pi_\mathcal{X}(B^\mathcal{X}_{\epsilon}(E)))>0$, which contradicts with the assumption that $\pi_{eq}$ is regular $K$ to one.
	\end{proof}
	
		\subsection{Sets satisfying interior saturation property} 
		We say that a set $E\in 2^X$ has  {\bf interior saturation property}  if 
		\begin{align}\label{defn-saturation}\pi_{eq}^{-1}(\text{int}(\pi_{eq}(E)))\subset E.
		\end{align}
		Denote by $\mathcal{E}_{isp}(X,G)$ the set of elements in $2^X$ having interior saturation property. 
		
	For convenience, we say that a minimal tds $(X,G)$ has {\bf  property $(*)$ } if it meets the conditions in Theorem \ref{main1'}. That is, one of the followings holds: 
	\begin{itemize}
		\item[(1)] $(X,G)$ is incontractible, i.e., the set of minimal points of $(X^n,G)$ is dense in $X^n$ for every $n\in\N$;
		\item[(2)] $(X, G)$ is local Bronstein, and admits an  invariant probability measure.
	\end{itemize}
	
	We now present the main result of this section, which asserts that a certain class of sets has interior saturation property. 
	\begin{prop}\label{prop-2}Let $(X,G)$ be a minimal tds with property $(*)$, and let $U \subset X$ be a nonempty open subset. Denote by  $M_{U}$
		the minimal center of $( \overline{Orb(\overline{U},G)},G)$. 
		Then, $$M_U\subset \mathcal{E}_{isp}(X,G).$$
	\end{prop}
	
Due to the technical nature of the proof of Proposition \ref{prop-2}, we defer it to the next section. Assuming the proposition holds, we now present a corollary that will be applied in Sections~\ref{IT-ab} and~\ref{s-mean-s-set}.
	\begin{lem}\label{lem-2}Let $(X,G)$ be a minimal tds with property $(*)$, and let $U\subset X$ be a nonempty open set with $B^{X_{eq}}_r(y_*)\subset \pi_{eq}(\overline{U})$ for some $r>0$ and $y_*\in X_{eq}$. Then, the following statements hold:
		\begin{itemize}
			\item[(1)]For every pair $(y,E)\in \overline{Orb((y_*,\overline{U}),G)}$, one has $B^{X_{eq}}_r(y)\subset \pi_{eq}(E)$.
			\medskip
			\item[(2)]For every $\epsilon\in(0,r)$,  one has
			$$\{g\in G: \pi_{eq}^{-1}(B^{X_{eq}}_{r-\epsilon}(gy_*))\subset \overline{B^X_\epsilon(gU)}\}\in\mathcal{F}_{ts}(G).$$
		\end{itemize}
	\end{lem}
	\begin{proof}For (1), let  $(y,E)\in \overline{Orb((y_*,\overline{U}),G)}$. There exists a sequence $\{g_n\}_{n=1}^\infty$ of $G$ such that $\lim_{n\to\infty} g_ny_*=y$ and  $\lim_{n\to\infty} g_n\overline{U}=E$. Since $\pi_{eq}$ is continuous, it follows that  $\lim_{n\to\infty} g_n\pi_{eq}(\overline{U})=\pi_{eq}(E)$. Note that $B^{X_{eq}}_r(y_*)\subset \pi_{eq}(\overline{U})$ and $(X_{eq},G)$ is an isometry. One has $B^{X_{eq}}_r(g_ny_*)\subset g_n\pi_{eq}(\overline{U})$ for each $n\in\N$, and thus, by letting $n\to\infty$, $B^{X_{eq}}_r(y)\subset \pi_{eq}(E)$.

		For	(2),
		given $\epsilon\in(0,r)$, let
		$$\mathcal{U}_\epsilon=\left\{(y,E)\in X\times 2^X: \min_{(y', E')\in M_{y_*,U}}\{d_{eq}(y,y')+d_H(E,E')\}<\epsilon\right\},$$
		where  $M_{y_*,U}$ is the minimal center of $(\overline{Orb((y_*,\overline{U}),G)},G)$. 
		Then, $\mathcal{U}_\epsilon$ is an open neighborhood of $M_{y_*,U}$. By Lemma \ref{syndetic} (3), one has	\begin{equation}\label{eq1234}
			\{g\in G: g(y_*,\overline{U})\in\mathcal{U}_\epsilon\}\in\mathcal{F}_{ts}(G).
		\end{equation}
	Thus, it suffices to show that
	\[
	\{g \in G : \pi_{eq}^{-1}(B^{X_{eq}}_{r - \epsilon}(g y_*)) \subset \overline{B^X_\epsilon(g U)}\} \supset \{g \in G : g(y_*, \overline{U}) \in \mathcal{U}_\epsilon\}.
	\]
	To this end, it is enough to prove that for each $(y, E) \in \mathcal{U}_\epsilon$, the inclusion
	\[
	\pi_{eq}^{-1}(B^{X_{eq}}_{r - \epsilon}(y)) \subset \overline{B^X_\epsilon(E)}
	\]
	holds.
	
Indeed, fix $(y, E) \in \mathcal{U}_\epsilon$. Then there exists $(y', E') \in M_{y_*, U}$ such that
\[
d_{eq}(y, y') + d_H(E, E') < \epsilon.
\]		Thus, by Lemma \ref{lem-2} (1), one has
		\begin{equation}\label{eq:188.55}
			\pi_{eq}^{-1}(B^{X_{eq}}_{r-\epsilon}(y))\subset \pi_{eq}^{-1}(B^{X_{eq}}_{r}(y'))\subset \pi_{eq}^{-1}(\text{int}(\pi_{eq}(E'))).
		\end{equation}
		Note that $E'$ belongs to the minimal center of $(\overline{Orb(\overline{U},G)},G)$. Then by Proposition \ref{prop-2}, we know that $E'$ has interior saturation property, which together with \eqref{eq:188.55}, implies that  
		$$	\pi_{eq}^{-1}(B^{X_{eq}}_{r-\epsilon}(y))\subset E'\subset B^{X}_\epsilon(E) .$$
		The proof is completed.
	\end{proof}
	\subsection{Proof of Proposition \ref{prop-2}}Let $(X,G)$ be a minimal tds.
		For any $F\in 2^X$, denote by 
	\begin{align}\label{eq-def-vf}
		V(F)=\{y\in X_{eq}:\exists E\in \mathcal{X}_{eq}^{meas}, E\subset F, \pi_{\mathcal{X}}(E)=y \}.\end{align}
	We now state the following criterion for the interior saturation property, which will be used to prove Proposition~\ref{prop-2}.
	\begin{lem}\label{lem-s-1}  If $F\in 2^X$ satisfies 
		\begin{align}\label{**}
			V(F)=\pi_{eq}(F),
		\end{align} then $F\in \mathcal{E}_{isp}(X,G)$.
	\end{lem}
	\begin{proof}We fix $F\in 2^X$ satisfying \eqref{**}, and prove that $F\in \mathcal{E}_{isp}(X,G)$.
		Denote by $Z_0$ the set of continuity points of the map $\pi_{eq}^{-1}$. Since $\pi_{eq}^{-1}$ is upper semi-continuous, by Lemma \ref{usc}, $Z_0$ is a residual subset of $X_{eq}$. It is not difficult to check that $\pi_{\mathcal{X}}^{-1}(y)=\{\pi_{eq}^{-1}(y)\}$ for $y\in Z_0$. Since $F$ satisfies \eqref{**},
		\begin{align}\label{34}\pi_{eq}^{-1}(Z_0\cap \text{int}(\pi_{eq}(F)))\subset F.
		\end{align}
		
		We claim that 
		\begin{align}\label{355}\pi_{eq}^{-1}(Z_0\cap \text{int}(\pi_{eq}(F)))\text{ is dense in }\pi_{eq}^{-1}(\text{int}(\pi_{eq}(F))). \end{align}If this is not true, we can find a nonempty open set $W\subset \pi_{eq}^{-1}(\text{int}(\pi_{eq}(F)))$ such that 
		\begin{align}\label{eq00}
			W\cap \pi_{eq}^{-1}(Z_0\cap \text{int}(\pi_{eq}(F)))=\emptyset. 
		\end{align}
		Thus, 
		$$\pi_{eq}(W)\cap Z_0\cap \text{int}(\pi_{eq}(F))=\emptyset.$$ 
		Since $\pi_{eq}(W)\subset \text{int}(\pi_{eq}(F))$, it follows that 
		$\pi_{eq}(W)\cap Z_0=\emptyset.$
		This is impossible, as $\pi_{eq}$ being semi-open implies $\pi_{eq}(W)$ has nonempty interior, and $Z_0$ is dense in $X_{eq}$. Therefore, \eqref{355} holds.
		
		Combining  \eqref{34}, \eqref{355} and the fact  $F$ is a closed set, one obtains that
		\begin{align*}\pi_{eq}^{-1}(\text{int}(\pi_{eq}(F)))\subset F.
		\end{align*}
		Thus, $F\in 	\mathcal{E}_{isp}(X,G)$.
	\end{proof}	
	Let 
		$$\mathcal{V}=\{F\in 2^X:	V(F)=\pi_{eq}(F)\}.$$
The preceding lemma shows that $\V \subset \mathcal{E}_{isp}(X,G)$. In what follows, we establish several properties of the set $\V$ that will be used in the proof of Proposition \ref{prop-2}.		
	\begin{lem}\label{lem-s-3}	The set
	$\V$ is a compact, $G$-invariant subset of $2^X$.
	\end{lem}
	\begin{proof} 
	The $G$-invariance of $\mathcal{V}$ is clear, so it remains to establish its compactness.  
	Suppose there exists a sequence $\{E_i\}_{i=1}^\infty \subset 2^X$ such that $V(E_i) = \pi_{eq}(E_i)$ for all $i \in \mathbb{N}$ and
	\begin{align}\label{e-0}
		\lim_{i \to \infty} E_i = E.
	\end{align}
	Our goal is to show that $V(E) = \pi_{eq}(E)$, i.e., $E \in \mathcal{V}$, which would imply the compactness of $\mathcal{V}$.
It is clear that $V(E) \subset \pi_{eq}(E)$, so it suffices to prove the reverse inclusion: $\pi_{eq}(E) \subset V(E)$.
		
		By \eqref{e-0},
	$\lim_{i\to\infty}\pi_{eq}(E_i)=\pi_{eq}(E)$, which implies that
	for any $y\in \pi_{eq}(E)$, there exist $y_i\in \pi_{eq}(E_i)$, $i\in\N$ such that 
		\begin{align}\label{e-2}\lim_{i\to\infty}y_i=y.\end{align}
		By 	\eqref{eq-def-vf},  for each $i\in\N$, there exists $E_i'\in\mathcal{X}_{eq}^{meas}$ such that
		\begin{align}\label{e-5}\pi_{\mathcal{X}}(E_i')=y_i,\end{align}
		and 	\begin{align}\label{e-5-1}E_i'\subset E_i.\end{align}
		Passing by a subsequence, we can assume that 
		\begin{align}\label{e-3}\lim_{i\to\infty}E_i'=E'\in\mathcal{X}_{eq}^{meas}.\end{align}
	
		Combining \eqref{e-2}, \eqref{e-5}, and \eqref{e-3} yields $y=\pi_{\mathcal{X}}(E')$,  and from \eqref{e-0}, \eqref{e-5-1}, and \eqref{e-3}, we deduce $E'\subset E$. Thus $y\in   V( E)$ and so by the arbitrariness of $y\in\pi_{{eq}}(E)$, one has   $\pi_{eq}(E)\subset V( E)$. Combining this with the argument at the beginning of the proof, we complete the proof.
	\end{proof}

	To apply the criterion in proving Proposition \ref{prop-2}, we first establish the following preparatory result.
	\begin{prop}\label{prop-1}Let $(X,G)$ be a minimal tds with property $(*)$, and let $U\subset X$ be a nonempty open subset.  Assume $M\subset \overline{Orb(\overline{U},G)}$ is a minimal set. Then,
		$$M\subset \{E\in 2^X: \exists E'\in \mathcal{X}^{meas}_{eq},E'\subset E\}.$$
	\end{prop}
	\begin{proof}Let  
		$$\mathcal{M}=\{ E\in 2^X: \exists E'\in\mathcal{X}_{eq}^{meas}, E'\subset E\}.$$
		By the compactness and $G$-invariance of $\mathcal{X}_{eq}^{meas}$, it is straightforward to see that $\mathcal{M}$ is a compact and $G$-invariant subset of $2^X$. Combining this with the minimality of $M$, to prove Proposition \ref{prop-1}, it suffices to show the existence of some $E_0\in M$ such that $E_0\in \mathcal{M}$.
		Furthermore, we only need to prove for any $\epsilon > 0$, there exists $E_\epsilon\in M$ such that 	\begin{align}\label{p-1}
			E_\epsilon\in \{E\in 2^X: \exists E'\in \mathcal{X}^{meas}_{eq},E'\subset B_{2\epsilon}^X(E)\}.
		\end{align}
		Indeed,	once this is  established, by letting $\epsilon\searrow 0$ and invoking the compactness of $ \mathcal{X}_{eq}^{meas}$ and $M$, we can obtain the existence of some $E_0\in M$ such that $E_0\in \mathcal{M}$.

		We now  prove \eqref{p-1} for a fixed $\epsilon > 0$. We have the following claim.
	\begin{cl}\label{cl1}The set
		\begin{align}\label{eq-syn-1}\mathcal{E}_{\epsilon}^X=\{g\in G: \exists E'\in \mathcal{X}^{meas}_{eq},E'\subset B_\epsilon^X(g\overline{U})\}\in\mathcal{F}_s(G).
	\end{align}
	\end{cl}
	We postpone the proof of the claim and first assume its validity to complete the proof of Proposition \ref{prop-1}.  Since $M\subset \overline{Orb(\overline{U},G)}$ is a minimal set, by Lemma \ref{syndetic} (4),
	\begin{align}\label{eq-syn-2}
		\mathcal{E}^\mathcal{X}_{\epsilon}=\{g\in G: g\overline{U}\in B^\mathcal{X}_\epsilon(M)\}\in\mathcal{F}_t(G).
		\end{align}
	Combining \eqref{eq-syn-1}, \eqref{eq-syn-2} and  Lemma \ref{F} (2), we can choose $g_\epsilon\in \mathcal{E}^\mathcal{X}_{\epsilon}\cap \mathcal{E}_\epsilon^X$.  Thus, there exists $E_\epsilon'\in \mathcal{X}_{eq}^{meas}$ such that 
	\[E_\epsilon'\overset{\eqref{eq-syn-1}}\subset B_\epsilon^X(g_\epsilon\overline{U})\overset{\eqref{eq-syn-2}}\subset  B^\mathcal{X}_{2\epsilon}(M),\]
	and hence there exists  $E_\epsilon\in M$ such that  
	$E_\epsilon'\subset  B_{2\epsilon}^X(E_\epsilon).$  Combined with the discussion at the beginning of the proof, this completes the proof of Proposition \ref{prop-1}.
	\end{proof}
It remains to complete the proof of Claim \ref{cl1}. 
\begin{proof}[Proof of Claim \ref{cl1}]
	Fix $E_*\in\mathcal{X}_{eq}^{meas}$, and let $y_*=\pi_\mathcal{X}(E_*)$. Then,
	\begin{align}\label{eq-1}E_*\subset \pi_{eq}^{-1}(y_*).
	\end{align}
	Let $\{x_i\}_{i=1}^n \text{ be an }({\epsilon}/{2})\text{-dense subset of }\pi^{-1}_{eq}(y_*).$
	Due to Remark \ref{rem-1}, both cases of property $(*)$ satisfy the local Bronstein condition. Using Lemma \ref{lem-rp} and Theorem \ref{rp}, there exists an open subset $V_i$ for $1\le i\le n$ such that 
	\begin{align}\label{eq-3-1}g_*V_i\subset U,
	\end{align}
	and 
	\begin{align*} V_i\subset B_{\epsilon/2}^X(x_i).
	\end{align*}
	The later implies 
	\begin{align}\label{eq-2}
		\pi_{eq}^{-1}(y_*)\subset B^X_\epsilon(\{z_1,\cdots,z_n\})\text{ whenever }(z_1,\cdots,z_n)\in V_1\times\cdots\times V_n.
	\end{align}
	
We now proceed to consider the two cases in the property $(*)$ separately.
In the case where $(X^n, G)$ has dense minimal points, we can find a minimal point $(x_1', \ldots, x_n')$ of $(X^n, G)$ contained in $V_1 \times \cdots \times V_n$. By Lemma \ref{syndetic} (1),
	$$\mathcal{S}_1=\{g\in G:g(x_1',\cdots,x_n')\in  V_1\times V_2\times \cdots\times V_n\}\in\mathcal{F}_s(G).$$
	Hence, for each $g\in \mathcal{S}_1$,
	$$E_*\overset{\eqref{eq-1}}\subset \pi_{eq}^{-1}(y_*)\overset{\eqref{eq-2}}\subset B^X_\epsilon(g\{x_1',\cdots,x_n'\})\overset{\eqref{eq-3-1}}\subset B^X_\epsilon(gg_*^{-1}\overline{U}).$$
	Thus, $\mathcal{E}_{\epsilon}^X$ is a syndetic set as it contains a syndetic set $\mathcal{S}_1g_*^{-1}$.
	
In the case where $(X, G)$ admits an ergodic probability measure $\mu$, we proceed as follows.  Since  $(X,G)$ is minimal, it follows that  $\mu$ has full support. Then, $$(\mu\times\mu\times\cdots\times\mu)(V_1\times V_2\times \cdots\times V_n)>0.$$
	Thus, by Lemma \ref{syndetic} (2),	\begin{align*}
		\mathcal{S}_2&=\{g\in G:g(V_1\times V_2\times\cdots\times V_n)\cap (V_1\times V_2\times \cdots\times V_n)\neq\emptyset \}\in\mathcal{F}_s(G).\end{align*}
	Hence, for $g\in\mathcal{S}_2$,
	$$E_*\overset{\eqref{eq-1}}\subset \pi^{-1}_{eq}(y_*)\overset{\eqref{eq-2}}\subset B^X_\epsilon(\cup_{i=1}^n((gV_i)\cap V_i))\overset{\eqref{eq-3-1}}\subset B^X_\epsilon(gg_*^{-1}\overline{U}).$$
	Thus, $\mathcal{E}_{\epsilon}^X$ is a syndetic set as it contains a syndetic set $\mathcal{S}_2g_*^{-1}$. Hence, the proof of Claim \ref{cl1} is complete.
\end{proof}
	
	Now we are able to prove Proposition \ref{prop-2}.
	\begin{proof}[Proof of Proposition \ref{prop-2}] To prove that
		\[
		M_U = \overline{Min\left(\overline{Orb(\overline{U}, G)}, G\right)} \subset \mathcal{E}_{isp}(X, G),
		\]
		it suffices, by Lemmas~\ref{lem-s-1} and~\ref{lem-s-3}, to show that every $E_* \in Min(\overline{Orb(\overline{U}, G)}, G)$ satisfies $V(E_*) = \pi_{eq}(E_*)$.
		To this end, it is enough to show that for any $\epsilon > 0$, the set $V(E_*)$ is $\epsilon$-dense in $\pi_{eq}(E_*)$. Since both $V(E_*)$ and $\pi_{eq}(E_*)$ are closed subsets of $X_{eq}$, this implies $V(E_*) = \pi_{eq}(E_*)$ by letting $\epsilon \searrow 0$.
		
		Now, we fix $\epsilon>0$ and  $E_*\in Min(\overline{Orb(\overline{U},G)},G)$ to prove the set $V(E_*)$ is $\epsilon$-dense in $\pi_{eq}(E_*)$. Let $\{y_1,\cdots,y_I\}\subset \pi_{eq}(U)$ be a finite $\epsilon/2$-dense subset of $\pi_{eq}(\overline{U})$.  Then,
		$\pi_{eq}(\overline{U})\subset \cup_{i=1}^IB_{\epsilon/2}^{X_{eq}}(y_i)$. For $i\in\{1,2,\cdots,I\}$, denote $B_i=\pi_{eq}^{-1}(B_{\epsilon/2}^{X_{eq}}(y_i))\cap U$, which is a nonempty open subset of $X$.  Then, 
		\begin{align}\label{eq-5}\text{diam}(\pi_{eq}(\overline{B_i}))<\epsilon,
		\end{align} and 
		\begin{align}\label{eq6}\overline{U}=\cup_{i=1}^I\overline{B_i}.
		\end{align}
		
		We consider a tds $(\overline{Orb((\overline{U}, \overline{B_1},\cdots,\overline{B_I}),G)},G)$. Let   $$p:(\overline{Orb((\overline{U}, \overline{B_1},\cdots,\overline{B_I}),G)},G)\to(\overline{Orb(\overline{U},G)},G)$$ be the  factor map to the first coordinate, and  let $$p_i:(\overline{Orb((\overline{U}, \overline{B_1},\cdots,\overline{B_I}),G)},G)\to(\overline{Orb(\overline{B_i},G)},G)$$ be the factor map to  the $(i+1)$-th coordinate for each $i=1,\cdots,I$.
		
		Since $E_*\in Min(\overline{Orb(\overline{U},G)},G)$, there exist a minimal set $M$ of $(\overline{Orb(\overline{U},G)},G)$ such that $E_*\in M$, and  a minimal set $L$ of  $(\overline{Orb((\overline{U}, \overline{B_1},\cdots,\overline{B_I}),G)},G)$ such that $p(L)=M$.
		
		Fix
		$(E_*,E_{*,1},\cdots,E_{*,I})\in L.
		$
		Then by Lemma \ref{lem-0} (2) and \eqref{eq6}, we have
		\begin{align}\label{eq5}E_*=\cup_{i=1}^IE_{*,i}.
		\end{align}
		Since	 $E_{*,i}\in p_i(L)$ and   $p_i(L)$ is a minimal set, it follows from Proposition \ref{prop-1} that there exists an element 
		\begin{align}\label{eq-6}E_i'\in \mathcal{X}_{eq}^{meas}\text{ such that }E_i'\subset E_{*,i}\subset E_*\text{ for }i=1,\cdots,I. 
		\end{align}
		
		Consider the tds $(Y,G)=((\pi_{eq}\times\cdots\times\pi_{eq})(\overline{Orb((\overline{U}, \overline{B_1},\cdots,\overline{B_I}),G)}),G)$. Due to the fact that $\pi_{eq}$ is a factor map, it follows that  
		$$Y=\overline{Orb((\pi_{eq}(\overline{U}), \pi_{eq}(\overline{B_1}),\cdots,\pi_{eq}(\overline{B_I})),G)}.$$
		Thus, $(\pi_{eq}(E_*),\pi_{eq}(E_{*,1}),\cdots,\pi_{eq}(E_{*,I}))\in Y$.
		By \eqref{eq-5} and Lemma \ref{lem-0} (1), \begin{align}\label{eq-7}\text{diam}(\pi_{eq}(E_{*,i}))=\text{diam}(\pi_{eq}(B_{i}))<\epsilon\text{, for $i=1,\cdots,I$.}
		\end{align}
	By \eqref{eq5}, we have $\pi_{eq}(E_*) = \cup_{i=1}^I \pi_{eq}(E_{*,i})$. Combining this with \eqref{eq-6} and \eqref{eq-7}, we conclude that
	\[
	\{\pi_{\mathcal{X}}(E_i') : i = 1, \ldots, I\} \text{ is } \epsilon\text{-dense in } \pi_{eq}(E_*).
	\]
	Moreover, by \eqref{eq-6}, each $\pi_{\mathcal{X}}(E_i')$ belongs to $V(E_*)$, and hence $V(E_*)$ is $\epsilon$-dense in $\pi_{eq}(E_*)$.
	This completes the proof of Proposition~\ref{prop-2}, in view of the argument at the beginning of the proof.
	\end{proof}
	\section{IT-set}\label{IT-ab} 
	
	In this section, we prove Theorem \ref{main1'}. To achieve this, we establish a stronger result, Theorem \ref{main1}. The key lies in discovering a new approach (Claim \ref{C-11}) for constructing infinite independence sets.
	\subsection{Statement of main result and corollaries} Recall that 
	$$\mathcal{X}_{eq}^{\text{meas}}=\{E\in\mathcal{X}:\nu_{eq}(\pi_\mathcal{X}(B^\mathcal{X}_{\epsilon}(E))>0, \forall \epsilon>0\},$$ 
	where  $\mathcal{X}=\overline{\{\pi_{eq}^{-1}(y): y\in X_{eq}\}}\text{ and }\pi_\mathcal{X}: \mathcal{X}\to X_{eq},~E\mapsto\pi_{eq}(E)$. In Lemma \ref{lem-1}, we have proven that the set $\{y\in X_{eq}:\pi_{eq}^{-1}(y)\in\mathcal{X}_{eq}^{meas}\}$ is a dense $G_\delta$ subset of $X_{eq}$ with full $\nu_{eq}$-measure.
	Thus, Theorem \ref{main1'} is a direct corollary of the following theorem.
	\begin{thm}\label{main1}Let $(X,G)$ be a minimal tds with property $(*)$. Then, every $E_*\in\mathcal{X}_{eq}^{\text{meas}}$ is an IT-set.	
	\end{thm}
	The proof of Theorem \ref{main1} is intricate; therefore, we defer its presentation to Subsection \ref{subsection-proof-meas}. For now, we assume its validity to derive some corollaries.
	\begin{cor}\label{cor-B1}Let $(X,G)$ be a minimal tds with property $(*)$.
		Then, the following statements hold:
		\begin{enumerate}[(i)]
			\item  If $(X,G)$ has no essential $K$-IT-tuple for some $K\ge 2$,  then $(X,G)$ is regular $K'$ to one for some $1\le K'\le K-1$.
			\item If $\pi_{eq}$ is open, then $\pi_{eq}^{-1}(y)$ is an IT-set for every $y\in X_{eq}$.
			\item Let $N\in\mathbb{N}\cup\{\infty\}$ and $ K\in\mathbb{N}$ with $2\le K\le N$. 	If $\pi_{eq}$ is almost $N$ to one, then for every $y\in X_{eq}$, $\pi_{eq}^{-1}(y)$ contains essential $K$-IT-tuples.
			\item If $(X,G)$ has no essential $K$-IT-tuples for some $K\ge 2$, then $|\mathcal{M}^e(X,G)|\le K-1$.
		\end{enumerate}
	\end{cor}
	\begin{proof} (i) follows directly from Theorem \ref{main1} and Lemma \ref{lem-1}.
	If $\pi_{eq}$ is open, by Lemma \ref{lem-per-xmeas} (2),  $\mathcal{X}_{eq}^{meas} = \{\pi_{eq}^{-1}(y) : y \in X_{eq}\}$. Hence, (ii) follows from Theorem \ref{main1}.
		(iii) is established using Theorem \ref{main1} and Lemma \ref{lem-1.2}.
		
		Now we prove (iv) by contradiction. Suppose that there exist distinct ergodic probability measures   $\mu_1,\mu_2,$ $\cdots, \mu_{K}$ of $(X,G)$. Then there exist  pairwise disjoint measurable subsets $V_1,V_2,\cdots, V_{K}$ of $X$ such that 
		$$\mu_k(V_k)\ge 1-\frac{1}{K+1}\text{ for }k=1,2,\cdots, K.$$
		Then,
		$$\nu_{eq}\left(\cap_{k=1}^{K} \pi_{eq}(V_k)\right)\ge \frac{1}{K+1}>0.$$
		By Theorem \ref{main1} and Lemma \ref{lem-1}, there exists $y_*\in \cap_{k=1}^{K} \pi_{eq}(V_k)$ such that $\pi_{eq}^{-1}(y_*)$ is an IT-set. Since  $V_1,V_2,\cdots, V_{K}$ are pairwise disjoint, it follows that  $|\pi_{eq}^{-1}(y_*)|\ge K$. Take distinct points $x_1,\cdots,x_{K}\in \pi_{eq}^{-1}(y_*)$. Then, $(x_1,\cdots,x_K)$ is an essential $K$-IT-tuple since $\pi_{eq}^{-1}(y_*)$ is an IT-set, which is a contradiction. We finish the proof.
	\end{proof}	
	\subsection{Proof of Theorem \ref{main1}}\label{subsection-proof-meas}Before proceeding to the proof of Theorem~\ref{main1}, we establish the following two lemmas.
	\begin{lem}\label{ll}
		Let $H$ be a compact metrisable topological group, and let $\nu_H$ be the  Haar measure on $H$. Assume $A$ is a measurable subset of $H$ and  $h_1,h_2,\cdots,h_M\in H$ and $\kappa\ge 0$. If $\nu_H(\cap_{m=1}^MAh_m^{-1})>\kappa$ then there exist  open neighborhoods $V_m$ of $h_m$ for all $m=1,2,\cdots,M$ such that 
		$$\nu_H(\cap_{m=1}^MA h_m'^{-1})>\kappa\text{ whenever } h_m'\in V_m.$$
	\end{lem}
	\begin{proof}By assumption, there exists a compact set $A_*\subset\cap_{m=1}^MAh_m^{-1}$ with $\nu_H(A_*)=\kappa_*>\kappa$. Thus, there exists an open neighborhood $W$ of $A_*$ such that $\nu_H(W\setminus A_*)<\frac{1}{M}(\kappa_*-\kappa)$. By the continuity of the map $G\times 2^{X_{eq}}\to 2^{X_{eq}}$ given by $(g,A)\mapsto gA$, we may choose an open neighborhood $U$ of the identity in $H$ such that $A_* U \subset W$. Then,
		\begin{align}\label{eq-121}
			\nu_H((A_*h)\setminus A_*)<\frac{1}{M}(\kappa_*-\kappa), \text{ for all } h\in U. 
		\end{align}
		Since $V_m=Uh_m$  is an open neighborhood of $h_m$, and  for any $ h_m'\in V_m$, $m\in\{1,2,\cdots,M\}$, one has that 
		\begin{align*}
			\nu_H(\cap_{m=1}^MA h_m'^{-1})&= \nu_H(\cap_{m=1}^M(Ah_m^{-1} (h_m'h_m^{-1})^{-1}))\\
			&\ge\nu_H(\cap_{m=1}^M(A_* (h_m'h_m^{-1})^{-1})\\
			&\ge  \nu_H(A_*)-\sum_{1\le m\le M} \nu_H( A_*\setminus A_* (h_m'h_m^{-1})^{-1})\\
			&= \nu_H(A_*)-\sum_{1\le m\le M} \nu_H( (A_*h_m'h_m^{-1})\setminus A_* )\\
			&>\kappa.
		\end{align*}
		Here the last inequality is from the fact that $h_m'h_m^{-1}\in U$ and \eqref{eq-121}. 
	\end{proof}
	Let $E(X_{eq})$ be the Ellis semigroup of $(X_{eq},G)$. By Theorem \ref{thm-A1},  $E(X_{eq})$  is  a compact metrisable topological group and there is an open factor map $$\tilde \pi: (E(X_{eq}),G)\to(X_{eq},G).$$
	Denote by $\tilde \nu$  the Haar measure on $E(X_{eq})$. Then $\nu_{eq}=\tilde \nu\circ \tilde\pi^{-1}$.  The relation is provided in the following diagram:
	$$\xymatrix{
		&~ &(X,G) \ar[d]_{\pi_{eq}}              \\
		&(E(X_{eq}),\tilde\nu,G)\ar[r]^{\tilde\pi} & (X_{eq},\nu_{eq},G). &}$$	 	
		
	\begin{lem}\label{lem:gap}
		Let $U_1,\ldots,U_M\subset E(X_{eq})$ be open subsets, and $Z$ be a  measurable subset of $X_{eq}$. If $\cap_{m=1}^M(\tilde{\pi}^{-1}(Z)U_m^{-1})\neq \emptyset$, then
		\[\cap_{m=1}^MN(\tilde{\pi}(U_m),Z)\in\F_s,\]
		where $N(A,B)=\{g\in G: gA\cap B\neq\emptyset \}$ for nonempty subsets $A,B$ of $X_{eq}$. 
	\end{lem}
	\begin{proof}
		Note that 
		\begin{equation}\label{eq:2025441457}\begin{split}
				&\{g\in G:g\in\cap_{m=1}^M(\tilde{\pi}^{-1}(Z)U_m^{-1}) \}\\
				&=\{g\in G:gU_m\cap \tilde{\pi}^{-1}(Z)\neq\emptyset,m=1,\cdots,M \}\\
				&\subset \cap_{m=1}^MN(\tilde{\pi}(U_m),Z).
		\end{split}
		\end{equation}
	By  assumption, $\cap_{m=1}^M(\tilde{\pi}^{-1}(Z)U_m^{-1})$ is a nonempty open subset of $X_{eq}$. Combining this with the fact that $(E(X_{eq}),G)$ is minimal, one has \[
	\{g\in G:g\in\cap_{m=1}^M(\tilde{\pi}^{-1}(Z)U_m^{-1}) \} \in \mathcal{F}_s,
		\]
		which, together with \eqref{eq:2025441457}, completes the proof.
	\end{proof}

	Now we are able to prove Theorem \ref{main1}.
	\begin{proof}[Proof of Theorem \ref{main1}]
		Given any $ E_*\in\mathcal{X}_{eq}^{\text{meas}}$, let $K\in\N$ and  $x_1,x_2,\cdots,x_K\in E_*$.   Now we show for any $\epsilon_*>0$,  there exists an infinite independence set for  $(B^X_{\epsilon_*}(x_1),\cdots,B^X_{\epsilon_*}(x_K))$. This finishes the proof of Theorem \ref{main1}, by the arbitrariness of $x_1,x_2,\cdots,x_K\in E_*$.
		
		Given $\epsilon_*>0$, let $Z_*=\pi_\mathcal{X}(B_{\epsilon_*/2}^\mathcal{X}(E_*))$. Since $E_*\in\mathcal{X}_{eq}^{\text{meas}}$, it follows that
		$\nu_{eq}(Z_*)>0$. By the definition of Hausdorff metric, 
		\begin{equation*}
			E_*\subset B_{\epsilon_*/2}^X(E),\text{ for all }E\in B_{\epsilon_*/2}^\mathcal{X}(E_*).
		\end{equation*}
		Combining this with the fact that  $E\subset \pi_{eq}^{-1}(\pi_\mathcal{X}(E))$ for every $E\in\mathcal{X}$, we have
		\begin{align}\label{5}E_*\subset B_{\epsilon_*/2}^X(\pi^{-1}_{eq}(y))\text{ for }y\in Z_*.\end{align}

		Set $\tilde Z_*=\tilde\pi^{-1}(Z_*)$. Then we have the following claim.
		\begin{cl}\label{C-11}Let $M\in\mathbb{N}$ and suppose there exist nonempty open subsets $W_1,\cdots, W_M$ of $X$, and $\{h_m\}_{1\le m\le M}\subset E(X_{eq})$ satisfying the following conditions:
			\begin{itemize}
				\item[(P1)]$\tilde \pi(h_m)\in \pi_{eq}(W_m)$ for $m=1,2,\cdots,M$;
				\item[(P2)]$\tilde \nu( \cap_{m=1}^M(\tilde Z_*h_m^{-1}))>0.$
			\end{itemize}Then, for any finite set $F\subset G$, there exist $ g_*\in G\setminus F$, and nonempty open sets  $W_{m,k}\subset W_m$ for $1\le m\le M,1\le k\le K$ such that
			\begin{equation}\label{eq:202544}
				g_*W_{m,k}\subset B_{\epsilon_*}^X(x_{k}).
			\end{equation}
			Moreover, there exist $\{h_{m,k}\}_{1\le m\le M,1\le k\le K}\subset E(X_{eq})$ such that	
			\begin{itemize}
				\item[(P1')]$\tilde \pi(h_{m,k})\in \pi_{eq}(W_{m,k})$ for $m=1,2,\cdots,M$, $k=1,2,\cdots,K$;
				\item[(P2')]$\tilde \nu( \cap_{m=1}^M\cap_{k=1}^K(\tilde Z_*h_{m,k}^{-1}))>0$.
			\end{itemize}
		\end{cl}
		The proof of Claim \ref{C-11} is complex, so we will present it later. For now, we assume that this claim holds to complete the proof of Theorem \ref{main1}.
		
		Put $W_0=X$ and take $h_0\in E(X_{eq})$. It is clear that $\tilde \pi(h_0)\in \pi_{eq}(W_0)$, and $\tilde \nu(\tilde Z_*h_0^{-1})=\tilde \nu(\tilde Z_*)>0$. Thus, applying Claim \ref{C-11}, we can find  $ g_1\in G$, and nonempty open sets  $W_{k}\subset W_0\text{ with } g_1W_{k}\subset B_{\epsilon_*}^X(x_{k})$ for $1\le k\le K$, and $\{h_k\}_{1\le k\le K}\subset E(X_{eq})$  such that	\begin{itemize}
			\item[(1)] $\tilde \pi(h_k)\in\pi_{eq}(W_{k})$ for $k=1,2,\cdots,K$;
			\item[(2)] $\tilde \nu(\cap_{k=1}^K\tilde Z_*h_{k}^{-1})>0.$
		\end{itemize}
		We can repeat the aforementioned process since the above items meets the conditions in Claim \ref{C-11}. By induction, there exist  nonempty open sets $W_{a}\subset X$ for $a\in \bigcup_{n=1}^\infty\{1,\cdots,K\}^n$ and a sequence $(g_n)_{n\in\mathbb{N}}$ of $G$ such that for  every $n\in\mathbb{N}$ and $a=a_1a_2\cdots a_n\in \{1,\cdots,K\}^n$ the following conditions hold:
		\begin{itemize}
			\item[(1)] $ W_{ak}\subset W_a$ for $k\in\{1,\cdots,K\}$;
			\item[(2)] $g_nW_{a}\subset B_{\epsilon_*}^X(x_{a_n})$;
			\item[(3)] $g_n\notin\{g_1,\cdots,g_{n-1}\}$.
		\end{itemize}
		Thus, for $N\in\N$ and $(a_n)_{n=1}^N\in \{1,\cdots,K\}^N$,	$$\cap_{n=1}^Ng_n^{-1}B_{\epsilon_*}^X(x_{a_n})\supset W_{a_1a_2\cdots a_n}\neq\emptyset.$$  This implies that $\{g_n\}_{n=1}^\infty$ is an infinite independence set for $(B_{\epsilon_*}^X(x_1),\cdots,B_{\epsilon_*}^X(x_K))$, and ends the proof of Theorem \ref{main1}.
	\end{proof}
	
	At the end of this section, we prove Claim \ref{C-11}.
	\begin{proof}[Proof of Claim \ref{C-11}] By (P2) and Lemma \ref{ll}, there exists an open neighborhood 
		$\tilde V_m$
		of $h_m$ for each $ m=1,2,\cdots,M$ such that 
		\begin{align}\label{1}\tilde \nu( \cap_{m=1}^M\cap_{k=1}^K(\tilde Z_*h_{m,k}^{-1}))>0\text{ whenever }h_{m,k}\in \tilde V_m.\end{align}	
		
		For	$ m=1,2,\cdots,M$, let
		\begin{align}\label{eq-01}
			W_m'=\pi_{eq}^{-1}(\tilde \pi (\tilde V_m))\cap W_m.
		\end{align}
		Since $\tilde \pi$ is an open map, it follows that $\tilde \pi (\tilde V_m)$ is an open neighborhood 
		of $\tilde \pi(h_m)$, which together with (P1), implies that $W_m'\subset X$ is a nonempty open set. Since $\pi_{eq}$ is semi-open, there exist $r\in(0,\epsilon_*)$ and $y_m\in X_{eq}$ such that \begin{equation}\label{eq:201517}
			B^{X_{eq}}_r(y_m)\subset \pi_{eq}(W_m')\subset \tilde{\pi}(\tilde V_m),\text{ 	for $ m=1,2,\cdots,M$.  }
		\end{equation}
		
		Denote
		$$\mathcal{N}=\{g\in G:gB^{X_{eq}}_{r/2}(y_m)\cap  Z_*\neq\emptyset,m=1,\cdots,M\}.$$
		By \eqref{eq:201517}, for all $m=1,2,\ldots,M$,  there exists $\tilde{h}_m\in \tilde{V}_m$ such that $\tilde{\pi}(\tilde{h}_m)\in B_{r/2}^{X_{eq}}(y_m)$, and hence by \eqref{1}, $\cap_{m=1}^M\tilde{Z}_*(\tilde\pi^{-1}(B^{X_{eq}}_{r/2}(y_m)))^{-1}\neq \emptyset$.
		By Lemma \ref{lem:gap}, one has 
		\begin{align*}
			\mathcal{N}\in\mathcal{F}_s(G).
		\end{align*}
		
		Now we prove there exists $g_*\in G\setminus F$ and  nonempty open sets  $W_{m,k}\subset W_m$ for $1\le m\le M,1\le k\le K$ such that \eqref{eq:202544} holds.
		Let $$\F_m=\{g\in G: \pi_{eq}^{-1}(B^{X_{eq}}_{r/2}(gy_m))\subset \overline{B^X_{r/2}(gW_m')}\}\text{ for } m=1,\cdots,M.$$
		The equation \eqref{eq:201517} allows us to apply Lemma \ref{lem-2} (2) to derive
		$$\F_m\in\mathcal{F}_{ts}(G), \text{ for all } m=1,\cdots,M.$$ 
		Since $\mathcal{N}\in\mathcal{F}_{s}(G)$, it follows from Lemma \ref{F} that
		\begin{align*}
			\cap_{m=1}^M\F_m\cap \mathcal{N}\in\mathcal{F}_s(G).
		\end{align*} 
		In particular, it is an infinite set, and	so   we can find $g_*\in (\cap_{m=1}^M\F_m\cap \mathcal{N})\setminus F.$  Since $g_*\in\mathcal{N}$ and $X_{eq}$ is isometric, it follows that  
		\begin{align*}
			B^{X_{eq}}_{r/2}(g_*y_m)\cap Z_*\neq\emptyset. 
		\end{align*}
		Let $y_m'\in B^{X_{eq}}_{r/2}(g_*y_m)\cap Z_*$ for $m=1,2,\ldots,M$. Since $g_*\in \F_m$, it follows that  
		$ \pi_{eq}^{-1}(y_m')\subset \overline{B^X_{r/2}(g_*W_m')} .$
		Thus, by the assumption $0<r<\epsilon_*$, we obtain that
		\begin{equation*}
			\{x_1,x_2,\cdots,x_K\}\subset	E_*\overset{\eqref{5}}\subset  B_{\epsilon_*/2}(\pi_{eq}^{-1}(y_m'))\subset B_{\epsilon_*}^X(g_*W_m'),
		\end{equation*} 
		and hence $$B^X_{\epsilon_*}(x_k)\cap g_*W_m' \neq\emptyset,\text{ for $1\le m\le M,1\le k\le K$.}$$
		For $1\le m\le M,1\le k\le K$, let $W_{m,k}=(g_*^{-1}B^X_{\epsilon_*}(x_k))\cap W_m'$ so that $$g_* W_{m,k}\subset B^X_{\epsilon_*}(x_k),$$
		that is, \eqref{eq:202544} holds.

		Finally, we prove (P1') and (P2') to complete the proof of Claim \ref{C-11}. Indeed, as for each $1\le m\le M,1\le k\le K$,
		$$\pi_{eq}(W_{m,k})\subset\pi_{eq}(W_{m}')\overset{\eqref{eq:201517}}\subset \tilde \pi (\tilde V_m),$$ 
		it follows that there exists $h_{m,k}\in\tilde{V}_m$ such that $\tilde{\pi}(h_{m,k})\in \pi_{eq}(W_{m,k})$.
		Due to \eqref{1}, 	we have \begin{align*}\tilde \nu( \cap_{m=1}^M\cap_{k=1}^K(\tilde Z_*h_{m,k}^{-1}))>0.\end{align*}	
		So,  (P1') and (P2')  hold and the proof is completed.
	\end{proof}

	\section{Weakly mean-sensitive set}\label{s-mean-s-set}
	In this section, we prove Theorem \ref{mainB} and present additional corollaries of Theorem \ref{mainB}. We will continue to use the notations in Section \ref{sec-isp} and Section \ref{IT-ab} .
	\subsection{Proof of Theorem \ref{mainB} (i)} In this subsection, we establish Theorem \ref{main2}, which together with Lemma \ref{lem-1}, implies Theorem \ref{mainB} (i). Meanwhile, we provide some corollaries of  Theorem \ref{main2}.
	\begin{thm}\label{main2}Let $(X,G)$ be a minimal tds  with local Bronstein condition, where  $G$ is amenable. Then  each $E_*\in\mathcal{X}_{eq}^{\text{meas}}$ is a  weakly mean-sensitive set.	
	\end{thm}
	Before proving the main theorem in this subsection, we establish the following lemma.
	\begin{lem}\label{4.1}Let $G$ be  amenable, and $F\in \mathcal{F}_t(G)$. Then there exists a F{\o}lner sequence $\{F_n\}_{n=1}^\infty$ of $G$ such that $\cup_{n\in\N}F_n\subset F$.
	\end{lem}
	\begin{proof} By definition, it is straightforward to verify the fact that for any F{\o}lner sequence $\{\hat F_n\}_{\in\N}$ of $G$ and any sequence $\{g_n\}_{n\in\N}$ in $G$, $\{\hat F_ng_n\}_{\in\N}$ is also a F{\o}lner sequence.
		
		Now we fix an F{\o}lner sequence $\{\hat F_n\}_{n=1}^\infty$  of $G$. Since $F\in \mathcal{F}_t(G)$, there exists $g_n\in G$ such that $\hat F_ng_n\subset F$ for each $n\in\N$. By the above fact, $\{\hat{F}_n g_n\}_{n=1}^\infty$ is the required F{\o}lner sequence.
	\end{proof}
	Now we prove  Theorem \ref{main2}.
	\begin{proof}[Proof of  Theorem \ref{main2}]Given $ E_*\in\mathcal{X}_{eq}^{\text{meas}}$, let $K\in\N$ and $x_1,x_2,\cdots,x_K\in E_*$.  
		To show  $E_*$ is a  weakly mean-sensitive set, it is sufficient to show that for any $\epsilon_*>0$,	there exists $\delta>0$ such that for any nonempty open set $U\subset X$,	$$\overline{BD}\{g\in G: (gU)\cap B^X_{\epsilon_*}(x_k)\neq\emptyset\text{ for }k=1,2,\cdots,K\}\ge \delta.$$
		
Since $E_* \in \mathcal{X}_{eq}^{\text{meas}}$, the set
$
Z_* := \pi_\mathcal{X}(B_{\epsilon_*/2}^\mathcal{X}(E_*))
$
has positive $\nu_{eq}$-measure. In particular,
\begin{align}\label{08}
	E_* \subset B^X_{\epsilon_*/2}(\pi_{eq}^{-1}(y)) ,~ \text{for every } y \in Z_*.
\end{align}
In the following, we will show that $\nu_{eq}(Z_*) > 0$ serves as the desired $\delta$.

		Now we fix a nonempty open set $U\subset X$. Since $\pi_{eq}$ is semi-open, we can find $y_*\in X_{eq}$ and $r\in(0,\frac{\epsilon_*}{2})$ such that $B_{r}^{X_{eq}}(y_*)\subset \pi_{eq}(U)$.
		Since $(X_{eq},G)$ is an isometry, one has $B_{r/2}^{X_{eq}}(gy_*)\cap Z_*\neq\emptyset$ if and only if $gy_*\in B^{X_{eq}}_{r/2}(Z_*)$. Due to the unique ergodicity of $(X_{eq},G)$, 
		\begin{align}\label{q-0}\underline{BD}\{g\in G:B_{r/2}^{X_{eq}}(gy_*)\cap Z_*\neq\emptyset \}=\underline{BD}\{g\in G:gy_*\in B^{X_{eq}}_{r/2}(Z_*)\}\ge \nu_{eq}(Z_*).
		\end{align}
		By Lemma \ref{4.1} and Lemma \ref{lem-2} (2), 
		\begin{align}\label{eq3}\overline{BD}\{g\in G: \pi_{eq}^{-1}(B^{X_{eq}}_{r/2}(gy_*))\subset \overline{B^X_{r}(gU)}\}=1.
		\end{align}
		Let $$\mathcal{N}=\{g\in G: B_{r/2}^{X_{eq}}(gy_*)\cap Z_*\neq\emptyset\text{ and } \pi_{eq}^{-1}(B^{X_{eq}}_{r/2}(gy_*))\subset \overline{B^X_{r}(gU)}\}.$$ Together with \eqref{q-0} and \eqref{eq3}, one has
		\begin{align}\label{N}\overline{BD}(\mathcal{N})\ge \nu_{eq}(Z_*).\end{align}
		Notice that for every $g\in \mathcal{N}$, there exists $y(g)\in Z_*$ such that $\pi_{eq}^{-1}(y(g))\subset \overline{B^X_{r}(gU)}$. Since $r\in(0,\epsilon_*/2)$, the above implies
		$$E_*\overset{\eqref{08}}\subset B^X_{\epsilon_*/2}(\pi^{-1}_{eq}(y(g)))\subset B^X_{\epsilon_*/2}( \overline{B^X_{r}(gU)})\subset  B^X_{\epsilon_*}(gU).$$
		Noting that $x_k\in E_*$, one has 	for every $k\in\{1,2,\cdots,K\}$ and $g\in \mathcal{N}$,
		$$ gU\cap B_{\epsilon_*}(x_k)\neq\emptyset. $$
		By \eqref{N}, $(x_1,\cdots,x_K)$ is a  weakly mean-sensitive tuple, and so $E_*$ is a  weakly mean-sensitive set.
	\end{proof}
	
	The following corollary can be proved using exactly the same argument as in Corollary \ref{cor-B1}, with Theorem \ref{main1} replaced by Theorem \ref{main2}. So we omit the proof.
	\begin{cor}\label{cor-S1}Let $(X,G)$ be a minimal tds with local Bronstein condition, where $G$ is amenable.
		Then, the following statements hold:
		\begin{enumerate}[(i)]
			\item  If $(X,G)$ has no essential  weakly mean-sensitive $K$-tuples for some $K\ge 2$,  then $(X,G)$ is regular $K'$ to one for some $1\le K'\le K-1$.
			\item If $\pi_{eq}$ is open, then $\pi_{eq}^{-1}(y)$ is a  weakly mean-sensitive set for every $y\in X_{eq}$.
			\item Let $N\in\mathbb{N}\cup\{\infty\}$ and $ K\in\mathbb{N}$ with $2\le K\le N$. 	If $\pi_{eq}$ is almost $N$ to one, then for every $y\in X_{eq}$, $\pi_{eq}^{-1}(y)$ contains essential  weakly mean-sensitive $K$-tuples.
			\item If $(X,G)$ has no essential  weakly mean-sensitive $K$-tuple for some $K\ge 2$, then $|\mathcal{M}^e(X,G)|\le K-1$.
		\end{enumerate}
	\end{cor}
	
	\subsection{Proof of Theorem \ref{mainB} (ii) and (iii)}
	To prove Theorem \ref{mainB} (ii) and (iii), we need the following lemma.
	\begin{lem}\label{ss}
		Let $(X,G)$ be a minimal tds, where $G$ is amenable. If a $K$-tuple $(x_1,\dots,x_K)\in X^K$ is a  weakly mean-sensitive tuple, then for any open neighborhood $U_k$ of $x_k$ for $k=1,2,\cdots,K$, one has $\nu_{eq}(\cap_{k=1}^K\pi_{eq}(U_k))>0$.
	\end{lem}
	\begin{proof} Choose an  open neighborhood  $U_k'$ of $x_k$ with $\overline{U_k'}\subset U_k$ for $k=1,2,\cdots,K$.  By definition of   weakly mean-sensitive tuple, there is $\delta_*>0$ such that for any nonempty open set $U\subset X$,
		$$\overline{BD}\{g\in G: U_k'\cap g U\neq\emptyset\text{ for }k=1,2\cdots,K\}\ge \delta_*.$$
		This implies that for any nonempty open set $V\subset X_{eq}$,
		\begin{align}\label{eq-24}
			\overline{BD}\{g\in G: \pi_{eq}(U_k')\cap gV\neq\emptyset\text{ for }k=1,2\cdots,K\}\ge \delta_*.
		\end{align}
		
		We will prove that $\nu_{eq}(\cap_{k=1}^K\pi_{eq}(U_k))\ge\delta_*>0$, which proves Lemma \ref{ss}. Suppose, for the sake of contradiction, that
		$\nu_{eq}(\cap_{k=1}^K\pi_{eq}(U_k))<\delta_*,$
		and hence
		\begin{align}\label{eq1}\nu_{eq}(\cap_{k=1}^K\pi_{eq}(\overline{U_k'}))<\delta_*.
		\end{align}
	Since $(x_1,\dots,x_K)$ is a   weakly mean-sensitive tuple, it follows that	$\pi_{eq}(x_1)=\cdots=\pi_{eq}(x_K)$. And so, $\cap_{k=1}^K\pi_{eq}(\overline{U_k'})$ is not empty.	Since $\cap_{k=1}^K\pi_{eq}(\overline{U_k'})$ is compact and nonempty, by \eqref{eq1}, and a standard argument, there exists $\epsilon>0$ such that 
		\begin{align}\label{eq2}
			\nu_{eq}\left(\overline{B_\epsilon^{X_{eq}}(\cap_{k=1}^K
				\pi_{eq}(\overline{U_k'}))}\right)<\delta_*.\end{align}
		
		\begin{cl}\label{c-8}There exists $\epsilon_*>0$ such that for every $\epsilon_*'\in(0,\epsilon_*)$ and $y\in X_{eq}$,
			if $B_{\epsilon_*'}^{X_{eq}}(y)\cap \pi_{eq}(\overline{U_k'})\neq\emptyset$ for all  $k=1,2,\cdots,K$, then 
			$$ B_{\epsilon_*'}^{X_{eq}}(y)\cap B_\epsilon^{X_{eq}}(\cap_{k=1}^K
			\pi_{eq}(\overline{U_k'}))\neq\emptyset.$$
		\end{cl}
		\begin{proof}[Proof of Claim \ref{c-8}]If Claim \ref{c-8} does not hold, we can find $\epsilon_n\searrow 0$ as $n\to\infty$ and $\{y_n\}_{n\in\N}\subset X_{eq}$ such that 
			$B_{\epsilon_n}^{X_{eq}}(y_n)\cap \pi_{eq}(\overline{U_k'})\neq\emptyset$ for $k=1,2,\cdots,K$ and 
			\begin{align}\label{eq0}B_{\epsilon_n}^{X_{eq}}(y_n)\cap B_\epsilon^{X_{eq}}(\cap_{k=1}^K
				\pi_{eq}(\overline{U_k'}))=\emptyset.
			\end{align} 
			Denote $B_k= \pi_{eq}(\overline{U_k'})\setminus B_\epsilon^{X_{eq}}(\cap_{k=1}^K
			\pi_{eq}(\overline{U_k'}))$ for $k=1,2,\cdots,K$. It is clear that 
			\begin{equation}\label{eq:123}
			\cap_{k=1}^K B_k = \emptyset.
			\end{equation}
			By \eqref{eq0},  $B_{\epsilon_n}^{X_{eq}}(y_n)\cap \pi_{eq}(\overline{U_k'})\subset B_k$.	Choose $y_n^{(k)}\in B_{\epsilon_n}^{X_{eq}}(y_n)\cap \pi_{eq}(\overline{U_k'})\subset B_k$ for $n\in\N$ and $k=1,2,\cdots,K$.  Passing to a subsequence, we assume that  $y_n^{(k)}\to y^{(k)}$ as $n\to\infty$ for $k=1,2,\cdots,K$. Since  $y_n^{(k)}\in B_k$ and $B_k$ is  compact, one has $y^{(k)}\in B_k$  for $k=1,2,\cdots,K$. Noting that $\epsilon_n\searrow 0$ as $n\to\infty$ and $\{y_n^{(k)}:k=1,2,\cdots,K\}\subset B_{\epsilon_n}^{X_{eq}}(y_n)$, one has
			$$y^{(1)}=y^{(2)}=\cdots=y^{(K)}.$$
			Thus, $\cap_{k=1}^K B_k \neq \emptyset$, which contradicts \eqref{eq:123}. Hence, Claim \ref{c-8} holds.
		\end{proof}
		Let $\epsilon_*>0$ be as in Claim \ref{c-8}. By \eqref{eq2} and a standard argument, we can find $\epsilon_*'\in(0,\epsilon_*)$ such that 
		\begin{align}\label{eq-012}
			\nu_{eq}(\partial(B_{\epsilon_*'+\epsilon}^{X_{eq}}(\cap_{k=1}^K
			\pi_{eq}(\overline{U_k'}))))=0,\end{align}
		and 
		\begin{align}\label{eq-22}	\nu_{eq}(B_{\epsilon_*'+\epsilon}^{X_{eq}}(\cap_{k=1}^K
			\pi_{eq}(\overline{U_k'})))<\delta_*.
		\end{align}
		Fixing $y_*\in X_{eq}$,  by the unique ergodicity of $(X_{eq}, G)$, \eqref{eq-012} and \eqref{eq-22}, we have
		\begin{align}\label{eq-23}
			\overline{BD}\{g\in G: gy_*\in B_{\epsilon_*'+\epsilon}^{X_{eq}}(\cap_{k=1}^K
			\pi_{eq}(\overline{U_k'}))\}<\delta_*. 
		\end{align}
		By Claim \ref{c-8}, if $B^{X_{eq}}_{\epsilon_*'}(y)\cap \pi_{eq}(\overline{U_k'})\neq\emptyset$ for all $k=1,2,\cdots,K$, then 
		$$ B_{\epsilon_*'}^{X_{eq}}(y)\cap B_\epsilon^{X_{eq}}(\cap_{k=1}^K
		\pi_{eq}(\overline{U_k'}))\neq\emptyset,$$
		and thus 
		$$y\in B^{X_{eq}}_{\epsilon_*'}(y)\subset B_{\epsilon_*'+\epsilon}^{X_{eq}}(\cap_{k=1}^K
		\pi_{eq}(\overline{U_k'})).$$
		Combining this with \eqref{eq-23} and the fact $(X_{eq},G)$ is an isometry, one has
		\begin{align*}
			&\overline{BD}\{g\in G: gB^{X_{eq}}_{\epsilon_*'}(y_*)\cap 
			\pi_{eq}(\overline{U_k'})\neq\emptyset\text{ for }k=1,2,\cdots,K\}\\
			&=\overline{BD}\{g\in G: B^{X_{eq}}_{\epsilon_*'}(gy_*)\cap 
			\pi_{eq}(\overline{U_k'})\neq\emptyset\text{ for }k=1,2,\cdots,K\}\\
			& \le \overline{BD}\{g\in G: gy_*\in B_{\epsilon_*'+\epsilon}^{X_{eq}}(\cap_{k=1}^K
			\pi_{eq}(\overline{U_k'}))\}\\
			&<\delta_*. 
		\end{align*}
	This contradicts \eqref{eq-24} (with $V = B^{X_{eq}}_{\epsilon_*'}(y_*)$), completing the proof of Lemma \ref{ss}.
	\end{proof}
	
	Now we prove Theorem \ref{mainB} (ii). 
	\begin{proof}[Proof of Theorem \ref{mainB} (ii)] Assume that $\pi_{eq}$ is regular $K$ to one. By Theorem \ref{mainB} (i), $(X,G)$ has essential  weakly mean-sensitive $K$-tuples. Now we show that $(X,G)$ does not have essential weakly mean-sensitive  $(K+1)$-tuples. 
		
	Suppose, for a contradiction, that there is an essential weakly mean-sensitive $(K+1)$-tuple  $(x_1,\cdots,x_{K+1})$.  We choose $\epsilon>0$ sufficiently small such that the sets $B^X_\epsilon(x_1),\cdots,B^X_{\epsilon}(x_{K+1})$ are pairwise disjoint. By  Lemma \ref{ss},
		\begin{align}\label{eq-010}\nu_{eq}(\cap_{k=1}^{K+1}\pi_{eq}(B^X_\epsilon(x_k)))>0.
		\end{align}
Thus,
\[
\nu_{eq}(\{y \in X_{eq} : |\pi_{eq}^{-1}(y)| = K+1\}) \ge \nu_{eq}\left(\cap_{k=1}^{K+1} \pi_{eq}(B^X_\epsilon(x_k))\right) > 0.
\]
This contradicts the assumption that $\pi_{eq}$ is regular $K$ to one. 
	\end{proof}
	
	Finally, we prove Theorem \ref{mainB} (iii).
	\begin{proof}[Proof of Theorem \ref{mainB} (iii)]Let $E_*$ be a  weakly mean-sensitive set  of $(X,G)$. Let $K\in\N$ and  $x_1,\cdots,x_K\in E_*$. 
		Let  $U_k$ be an open neighborhood of $x_k$ for $k=1,2,\cdots,K$. By Lemma \ref{ss}, $$\nu_{eq}(\cap_{k=1}^K\pi_{eq}(U_k))>0.$$ By Theorem \ref{main1'}, there exists $\tilde y\in \cap_{k=1}^K\pi_{eq}(U_k)$ such that $\pi^{-1}_{eq}(\tilde y)$ is an IT-set. By taking $\tilde x_k\in U_k\cap \pi^{-1}_{eq}(\tilde y)$ for $k=1,2,\cdots,K$, we have a $K$-IT-tuple $(\tilde x_1,\cdots,\tilde x_K)$, and so $(U_1,\cdots,U_K)$ has infinite independence sets.  Thus, by the arbitrariness of $(U_1,\cdots,U_K)$, one has $( x_1,\cdots, x_K)$ is an IT-tuple, and so $E_*$ is an IT-set.
	\end{proof}

	\noindent{\bf Acknowledgment.} 	The authors thank Professor Garc\'{\i}a-Ramos for his many valuable suggestions on earlier versions of this paper.
This paper is  supported by National Key R\&D Program of China (No. 2024YFA1013602, 2024YFA1013600) and National Natural Science Foundation of China grants (12031019, 12371197, 12426201).
	

\begin{thebibliography}{99}
		\bibitem{A} J. Auslander, A group theoretic condition in topological dynamics, Proceedings of the 18th Summer Conference on Topology and its Applications. Topology Proc. 28 (2004), no. 2, 327--334.
		
		\bibitem{A1} J. Auslander, Minimal flows and their extensions, Elsevier Science Ltd, 1988.  
		
		\bibitem{A3} J. Auslander, Minimal flows with a closed proximal cell,	Ergodic Theory Dynam. Systems 21 (2001), no. 3, 641--645.
		\bibitem{ber} V. Bergelson, N. Hindman and R. McCutcheon, Notions of size and combinatorial properties of quotient sets in semigroups, Proceedings of the 1998 Topology and Dynamics Conference (Fairfax, VA).
		Topology Proc. 23 (1998), Spring, 23--60.
		
		\bibitem{B1} F. Blanchard, Fully positive topological entropy and topological mixing, Symbolic dynamics and its applications (New Haven, CT, 1991), 95--105,
		Contemp. Math., 135, Amer. Math. Soc., Providence, RI, 1992.
		
		\bibitem{BHM}	F. Blanchard, B. Host and A. Maass, Topological complexity,	Ergodic Theory Dynam. Systems  20 (2000), 641--662.
		
		\bibitem{B2}F. Blanchard, B. Host, A. Maass, S. Martinez and D. J. Rudolph, Entropy pairs for a measure,
		Ergodic Theory Dynam. Systems 15 (1995), 621--632.
		
		\bibitem{CS} Y. Cao and S. Shao, Almost proximal extensions of minimal flows, Ergodic Theory Dynam. Systems 43 (2023), no. 12, 4041--4073.
		
		\bibitem{Vr}	J. de Vries, Elements of Topological Dynamics, Kluwer Academic Publishers (993),
		Dordrecht.
		
		
		\bibitem{EG} R. Ellis and W. Gottschalk, Homomorphisms of transformation groups, Trans. Amer. Math. Soc. 94,	1960, 258--271.
		\bibitem{F} E. F{\o}lner, On groups with full Banach mean value, Math. Scand. \textbf{3} (1955), 243--254.
		
		\bibitem{FG} G. Fuhrmann, E. Glasner, T. J\"ager and C. Oertel, Irregular model sets and tame dynamics, Trans. Amer. Math. Soc. 374 (2021), no. 5, 3703--3734.
		\bibitem{FuhrmannKwietniak2020}  G.	Fuhrmann and D. Kwietniak,  On tameness of almost automorphic dynamical systems for general groups, Bull. Lond. Math. Soc. 52, 1 (2020), 2--42.
		\bibitem{Furbook}H.	Furstenberg, 	Recurrence in ergodic theory and combinatorial number theory,M. B. Porter Lectures. Princeton University Press, Princeton, NJ, 1981. xi+203 pp.
		
		\bibitem{Felipe}F.~Garc\'{\i}a-Ramos,
		Weak forms of topological and measure-theoretical equicontinuity: relationships with discrete spectrum and sequence entropy. 
		Ergodic Theory Dynam. Systems 37 (2017), no. 4, 1211–1237.
		
		\bibitem{GH}F.~Garc\'{\i}a-Ramos and H. Li, Local entropy theory and applications, arXiv:2401.10012.
		
		\bibitem{G1} E. Glasner, The structure of tame minimal dynamical systems for general groups, Invent. Math. 211 (2018), no. 1, 213--244.
		
		\bibitem{G1996} E. Glasner,
		Structure theory as a tool in topological dynamics, Descriptive set theory and dynamical systems (Marseille-Luminy, 1996), 173--209,
		London Math. Soc. Lecture Note Ser. 277, Cambridge Univ. Press, Cambridge, 2000.
		\bibitem{G2005}	E. Glasner, 
		Topological weak mixing and quasi-Bohr systems,
		Probability in mathematics.
		Israel J. Math. 148 (2005), 277--304.
		\bibitem{G3}E. Glasner, On tame dynamical systems, Colloq. Math. 105 (2006), no. 2, 283--295.
		
		\bibitem{G75} S. Glasner, Compressibility properties in topological dynamics, Amer. J. Math. 97 (1975), 148--171.
		\bibitem{G2} E. Glasner and M. Megrelishvili, Circularly ordered dynamical systems, Monatsh. Math. 185 (2018), no. 3, 415--441.
		
		\bibitem{GW} E. Glasner and B. Weiss, Topological entropy of extensions, Ergodic theory and its connections with harmonic analysis (Alexandria, 1993), 299--307, London Math. Soc. Lecture Note Ser., 205, Cambridge Univ. Press, Cambridge, 1995.
		
		\bibitem{GY} E. Glasner and X. Ye, Local entropy theory, Ergodic Theory Dynam. Systems 29 (2009), 321--356
		
		
		
		
		\bibitem{GO2025} J. Gómez, I. León-Torres and V. Muñoz-López, Sequence entropy and independence in free and minimal actions, arXiv:2504.00960.
		
		
		\bibitem{Haar} E. Hewitt and K. A. Ross, Abstract harmonic analysis, Vol. I: Structure of topological groups. Integration theory,
		group representations. Die Grundlehren der mathematischen Wissenschaften, Bd. 115. Academic Press, Inc.,
		Publishers, New York; Springer-Verlag, Berlin-Göttingen-Heidelberg, 1963. 
		
		\bibitem{H1} W. Huang, Tame systems and scrambled pairs under an Abelian group action, Ergodic Theory Dynam. Systems 26 (2006), no. 5, 1549--1567.
		
		
		\bibitem{HLSY1} W. Huang, S. Li, S. Shao and X. Ye, Null systems and sequence entropy pairs, Ergodic Theory Dynam. Systems 23 (2003), no. 5, 1505--1523.
		
		\bibitem{HLY} W. Huang, P. Lu and X. Ye,
		Measure-theoretical sensitivity and equicontinuity,
		Israel J. Math. 183 (2011), 233--283.
		
		\bibitem{HLSY} W. Huang, Z. Lian, S. Shao and X. Ye, Minimal systems with finitely many ergodic measures, J. Funct. Anal. 280 (2021), no. 12, Paper No. 109000, 42 pp.
		
		\bibitem{HMY}W. Huang, A. Maass and X. Ye, Sequence entropy pairs and complexity pairs for a measure,
		Ann. Inst. Fourier (Grenoble), 54 (2004), 1005--1028.
		
		\bibitem{HY2} W. Huang and X. Ye, Topological complexity, return times and weak disjointness, Ergodic Theory Dynam. Systems, 24 (2004), 825--846.
		
		\bibitem{HY} W. Huang and X. Ye, A local variational relation and applications, Israel J. Math. 151 (2006),
		237--279.
		
		\bibitem{KL3} D. Kerr and H. Li,  Combinatorial independence and sofic entropy, Commun. Math. Stat. 1 (2013), no. 2, 213--257.
		
		\bibitem{KH} D. Kerr and H. Li, Independence in topological and $C^*$-dynamics, Math. Ann. 338 (2007), no. 4, 869--926.
		
		
		
		
		
		
		\bibitem{Ku} K. Kuratowski, Topology, Vol. I, Acad. Press, New York, N.Y., 1966.
		\bibitem{Ku2} K. Kuratowski, Topology, Vol. II, Acad. Press, New York, N.Y., 1968.
		
		\bibitem{LTY} J. Li, S. Tu and X. Ye, Mean equicontinuity and mean sensitivity, Ergodic Theory and Dynamical Systems 35 (2015), 2587--2612.
		
		\bibitem{LY} J. Li and Y. Yang, Stronger versions of sensitivity for minimal group actions, Acta Math. Sin. (Engl. Ser.) 37 (2021), no. 12, 1933--1946.
		\bibitem{LY21}	J.~Li and T. Yu,  On mean sensitive tuples,   J. Differential Equations,  {297} (2021), 175--200.
		\bibitem{LLTS2024}	J. Li, C. Liu, S. Tu and T. Yu,  Sequence entropy tuples and mean sensitive tuples, Ergodic Theory Dynam. Systems 44 (2024), no. 1, 184--203.
		
		\bibitem{LYY22}	J.~Li, X.~Ye and T.~Yu,  Equicontinuity and sensitivity in mean forms, J. Dynam. Differential Equations, {34} (2022), 133--154.
		\bibitem{LWX}	C. Liu, X. Wang and L. Xu,  Sequence entropy and IT-tuples for minimal group actions. Adv. Math. 467 (2025), Paper No. 110183. 
		
		\bibitem{MaassShao}	A. Maass and S. Shao, 
		Structure of bounded topological-sequence-entropy minimal systems,
		J. Lond. Math. Soc. (2) 76 (2007), no. 3, 702--718.
		
		\bibitem{Mc}D. McMahon, Relativized weak disjointness and relatively invariant measures,
		Trans. Amer. Math. Soc. 236 (1978), 225--237.
		
		\bibitem{Q} J. Qiu, Independence and almost automorphy of higher order, Ergodic Theory Dynam. Systems 43 (2023), no. 4, 1363--1381.
		
		\bibitem{QY} J. Qiu and J. Yu, Saturated theorem along cubes for a measure and applications, arXiv:2311.14198.
		
		\bibitem{Vee}	W. A. Veech, Point-distal flows, Amer. J. Math. 92 (1970) 205--242.
		\bibitem{Veech}	W. A. Veech, The equicontinuous structure relation for minimal Abelian transformation groups,
		Amer. J. Math. 90 (1968), 723--732.
		
		\bibitem{W} H. Wang, Thickly syndetic sensitivity of semigorup actions, Bull. Korean Math. Soc. 55 (2018), no. 4, 1125--1135.
		
		\bibitem{P} P. Walters, An Introduction to Ergodic Theory (Graduate Texts in Mathematic, 79), Springer-Verlag,
		New York, 1982.
		
		
		\bibitem{YZ} X. Ye and R. Zhang, On sensitive sets in topological dynamics, Nonlinearity 21 (2008),		1601--1620.
		
		
		
		
	\end{thebibliography}
\end{document}